\documentclass[a4paper,english,cleveref,nolineno]{socg-lipics-v2021} %

\usepackage{amsthm,amssymb,amsmath}
\usepackage{tikz,tikz-cd}
\usepackage[noend,noline]{algorithm2e}
\usepackage{mathtools}
\usepackage{cite}
\makeatletter
\newcommand{\citecomment}[2][]{\citen{#2}#1\citevar}
\newcommand{\citeone}[1]{\citecomment{#1}}
\newcommand{\citetwo}[2][]{\citecomment[,~#1]{#2}}
\newcommand{\citevar}{\@ifnextchar\bgroup{;~\citeone}{\@ifnextchar[{;~\citetwo}{]}}}
\newcommand{\citefirst}{\@ifnextchar\bgroup{\citeone}{\@ifnextchar[{\citetwo}{]}}}
\newcommand{\cites}{[\citefirst}
\makeatother

\newcommand{\op}{\operatorname}

\newcommand{\card}[1]{\operatorname{card}#1}
\newcommand{\pivot}[1]{\operatorname{PivInd}#1}
\newcommand{\pivots}[1]{\operatorname{PivInds}#1}
\newcommand{\chainpivot}[2]{\operatorname{Pivot}_{#1}#2}
\newcommand{\Flow}{\operatorname{F}}

\newcommand{\IP}[2]{\langle#1,#2\rangle}
\newcommand{\supp}[2]{\op{supp}_{#1}#2}
\newcommand{\spann}[1]{\op{span}#1}
\newcommand{\chainbasis}[1]{\Sigma_{#1}}
\newcommand{\pivotentry}[1]{\op{PivEnt}#1}
\newcommand{\VMatrix}{S}
\newcommand{\descpx}[3]{D_{#3}(#1,#2)}
\newcommand{\descpxVC}[2]{D(#2,#1)}
\newcommand{\descpxV}[1]{D(#1)}
\newcommand{\descpxf}[2]{D_{#2}(#1)}
\newcommand{\sublevel}[2]{S_{#2}(#1)}
\newcommand{\sublevelopen}[2]{S^{<}_{#2}(#1)}
\newcommand{\gradbound}[2]{\partial_{#1}#2}
\newcommand{\spacegradientfacets}[1]{L_{#1}}
\newcommand{\projfacets}[1]{\pi_{#1}}
\newcommand{\critsubset}{C}
\newcommand{\critsubsetparameter}[3]{\op{Crit}_{#1}(#2,#3)}
\newcommand{\gradres}[2]{{#1}_{#2}}
\newcommand{\del}[2][]{\op{Del}_{#1}(#2)}
\newcommand{\delo}[2][]{\op{Del}\!^<_{#1}(#2)}
\newcommand{\wrap}[2][]{\op{Wrap}_{#1}(#2)}
\newcommand{\cech}[2][]{\op{\text{\v{C}}ech}_{#1}(#2)}
\newcommand{\cecho}[2][]{\op{\text{\v{C}}ech}\!^<_{#1}(#2)}

\newtheorem{thmx}{Theorem}

\newtheorem{corx}[thmx]{Corollary}

\newcommand\vcent[1]{\begingroup
\setbox0=\hbox{#1}\parbox{\wd0}{\box0}\endgroup}

\title{Wrapping Cycles in Delaunay Complexes: Bridging Persistent Homology and Discrete Morse Theory}
\titlerunning{Wrapping Cycles in Delaunay Complexes}

\author{Ulrich Bauer}{Department of Mathematics and Munich Data Science Institute,\newline Technische Universität München, Germany \and
\url{www.ulrich-bauer.org}}{mail@ulrich-bauer.org}{https://orcid.org/0000-0002-9683-0724}{}%
 \author{Fabian Roll}{Department of Mathematics, Technische Universität München, Germany \and \url{https://www.roll.science}}{fabian.roll@tum.de}{https://orcid.org/0000-0002-3604-4545}{}%
\authorrunning{U. Bauer and F. Roll}
\Copyright{Ulrich Bauer and Fabian Roll} %

\category{} %
\funding{This research has been supported by the Deutsche Forschungsgemeinschaft (DFG -- German Research Foundation) -- Project-ID 195170736 -- TRR109 \emph{Discretization in Geometry and Dynamics}.}%
\acknowledgements{We thank Herbert Edelsbrunner and Marian Mrozek for interesting discussions about the connection between algebraic gradient flow in discrete Morse theory and matrix reduction in persistent homology.}

\supplement{Source code publicly available under MIT license.}
\supplementdetails[subcategory={Source code}, cite={git}, swhid={swh:1:dir:5edd13ee3b341318256c442f703d84a0b1b27c6a}]{Software}{https://github.com/fabian-roll/wrappingcycles}

\EventEditors{Wolfgang Mulzer and Jeff M. Phillips}
\EventNoEds{2}
\EventLongTitle{40th International Symposium on Computational Geometry (SoCG 2024)}
\EventShortTitle{SoCG 2024}
\EventAcronym{SoCG}
\EventYear{2024}
\EventDate{June 11-14, 2024}
\EventLocation{Athens, Greece}
\EventLogo{socg-logo.pdf}
\SeriesVolume{293}
\ArticleNo{XX}     %

\hideLIPIcs

\usepackage[marginpar]{todo}

\begin{document}
\maketitle
\begin{abstract}
We study the connection between discrete Morse theory and persistent homology in the context of shape reconstruction methods.
Specifically, we consider the construction of Wrap complexes, introduced by Edelsbrunner as a subcomplex of the Delaunay complex, and the construction of lexicographic optimal homologous cycles, also considered by Cohen–Steiner, Lieutier, and Vuillamy in a similar setting.
We show that for any cycle in a Delaunay complex for a given radius parameter, the lexicographically optimal homologous cycle is supported on the Wrap complex for the same parameter, thereby establishing a close connection between the two methods.
We obtain this result by establishing a fundamental connection between reduction of cycles in the computation of persistent homology and gradient flows in the algebraic generalization of discrete Morse theory.
\end{abstract}

\vspace{3em}

\section{Introduction}

Reconstructing shapes and submanifolds from point clouds is a classical topic in computational geometry.
Starting in the 2000s, several key results have been achieved \cite{MR1318786,MR1845986,MR1885501,MR2038483,MR2267420,MR2298361,MR2485260,MR3148657,MR2383768}, culminating in theoretical homeomorphic reconstruction guarantees for a method based on the Delaunay triangulation \cite{MR2298361}.
The method is theoretical in nature, and their practical applicability is hindered by complexity and robustness issues.
A major challenge is caused by \emph{slivers} \cite{MR1865932}, which are simplices in the Delaunay triangulation with small volume but no short edges, and which have to be handled explicitly.
In contrast, several related Delaunay-based methods have proven highly robust and successful in practice, in particular, Morse-theory based methods such as \emph{Wrap} and related constructions \cite{MR2038483,MR2460367,MR2485260,DBLP:journals/cgf/Sadri09,MR3605986, DBLP:journals/tog/PortaneriRHCA22} and homological methods based on minimal cycles \cite{cohensteiner:hal-02391190,MR4517098,MR4598045,MR4470887,vuillamy:tel-03339931}.
The latter methods produce \emph{water-tight} surfaces (that is, boundaries of solids) by construction, gracefully circumventing the issue of slivers.
On the other hand, the \emph{Wrap complex} \cite{MR3605986} is always homotopy-equivalent to a union of balls of a given radius, but may still contain critical sliver simplices. It is therefore desirable to identify a subcomplex that is free from slivers.

Even though the Morse-theoretic and the homological method seem closely related in spirit, up to now the connection between them has not been fully explored.
While the development of Wrap preceded the introduction of Forman's discrete Morse theory \cite{MR1612391}, it has since been redeveloped within this framework \cite{MR3605986}, which provides powerful tools for proving geometric and topological properties of this construction.
A bridge from Wrap to homology-based methods advances our understanding of their geometric and topological properties, potentially paving the way for reconstruction guarantees that extend beyond surfaces.
Indeed, a synthesis of Wrap, discrete Morse theory, and persistent homology in the context of surface reconstruction has been envisioned already in Edelsbrunner's original paper describing the Wrap algorithm \cite{MR2038483}.

\subparagraph*{Contributions} 

The objective of this paper is twofold. 
First, we provide a tight link between persistent homology and discrete Morse theory, unifying the central notions of persistence pairs and of gradient pairs in a common framework.
Despite various results connecting both theories (see, e.g., \cite{MR2872542,MR3090522,MR4298669}), this kind of interface has been missing from the literature.
Using a specific refinement of the sublevel set filtration of a discrete Morse function, we demonstrate how the persistence pairs yield an algebraic gradient that contains as a subset the gradient pairs of the discrete Morse function, and such that the corresponding algebraic gradient flow can be viewed as a variant of the reduction algorithm for computing persistent homology (see \cref{algmorse-persistence}).
Second, we use these insights to establish a strong connection between the aforementioned Morse-theoretic and homological approaches to shape reconstruction.
Our main result (specializing \cref{inv-chain-supp-on-desc-cpx} to the Delaunay radius function, see \cref{ex:cech-del}), shows that the lexicographically minimal cycles (see \cref{def:lexoptimal}) in a Delaunay complex are all supported on the corresponding Wrap complex (see \cref{{def:desc-cpx}}), as illustrated in~\cref{fig:2dreduction}.
In the statement of the theorem, the order on simplices follows the convention outlined in \cref{discrete-morse-apparent-pairs}, and homology is considered with coefficients in a specified field.
\begin{thmx}
  \label{theorem-a}
Let $X\subset \mathbb{R}^d$ be a finite subset in general position, let $r \in \mathbb R$, and let $h \in H_*(\del[r]{X})$ be a homology class of the Delaunay complex $\del[r]{X}$.
Then the lexicographically minimal cycle of $h$, with respect to the Delaunay-lexicographic order on the simplices, is supported on the Wrap complex $\wrap[r]{X}$.
\end{thmx}

\begin{figure}[t]
    \vcent{\includegraphics[height=.3\textwidth]{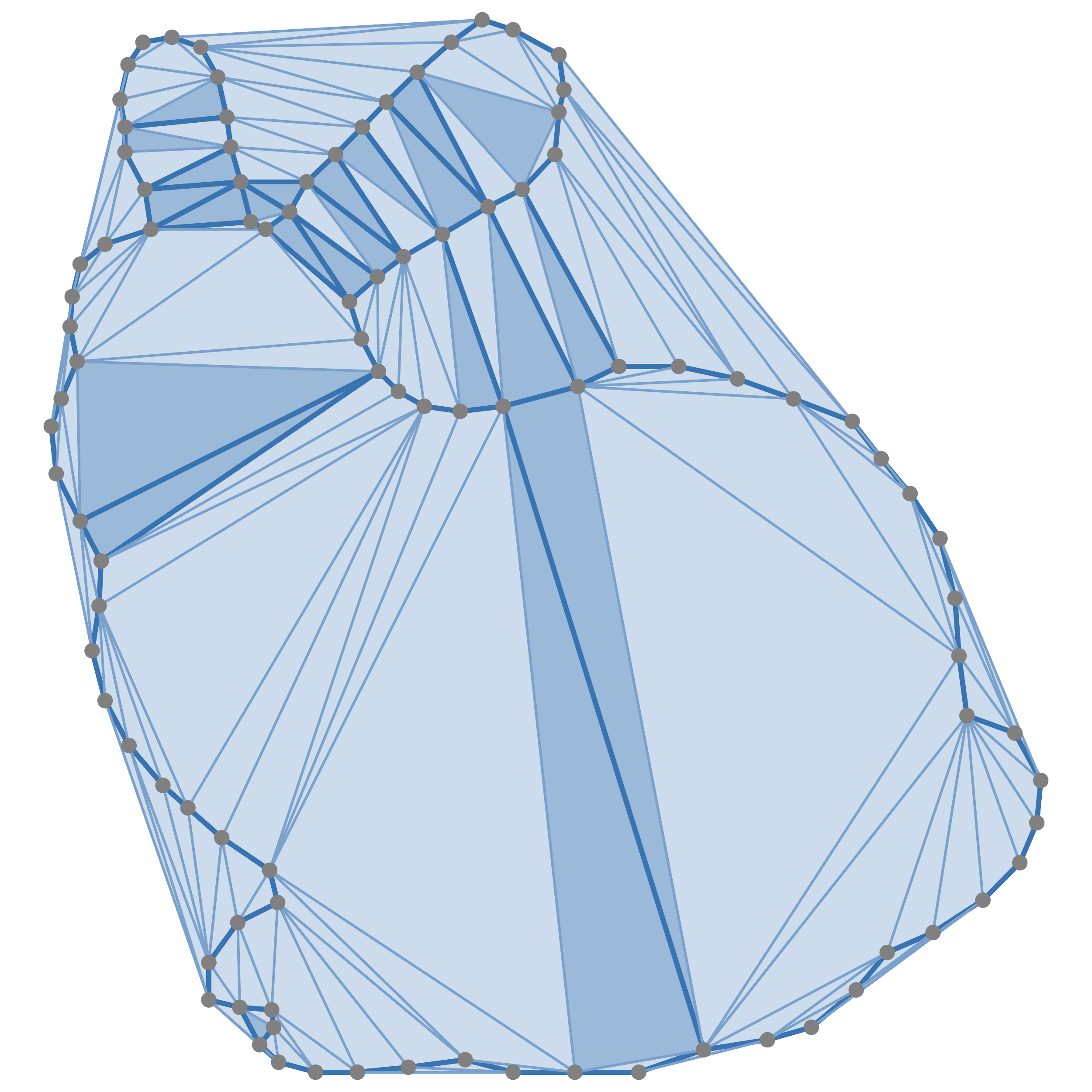}}
  \hfill
    \vcent{\includegraphics[height=.3\textwidth]{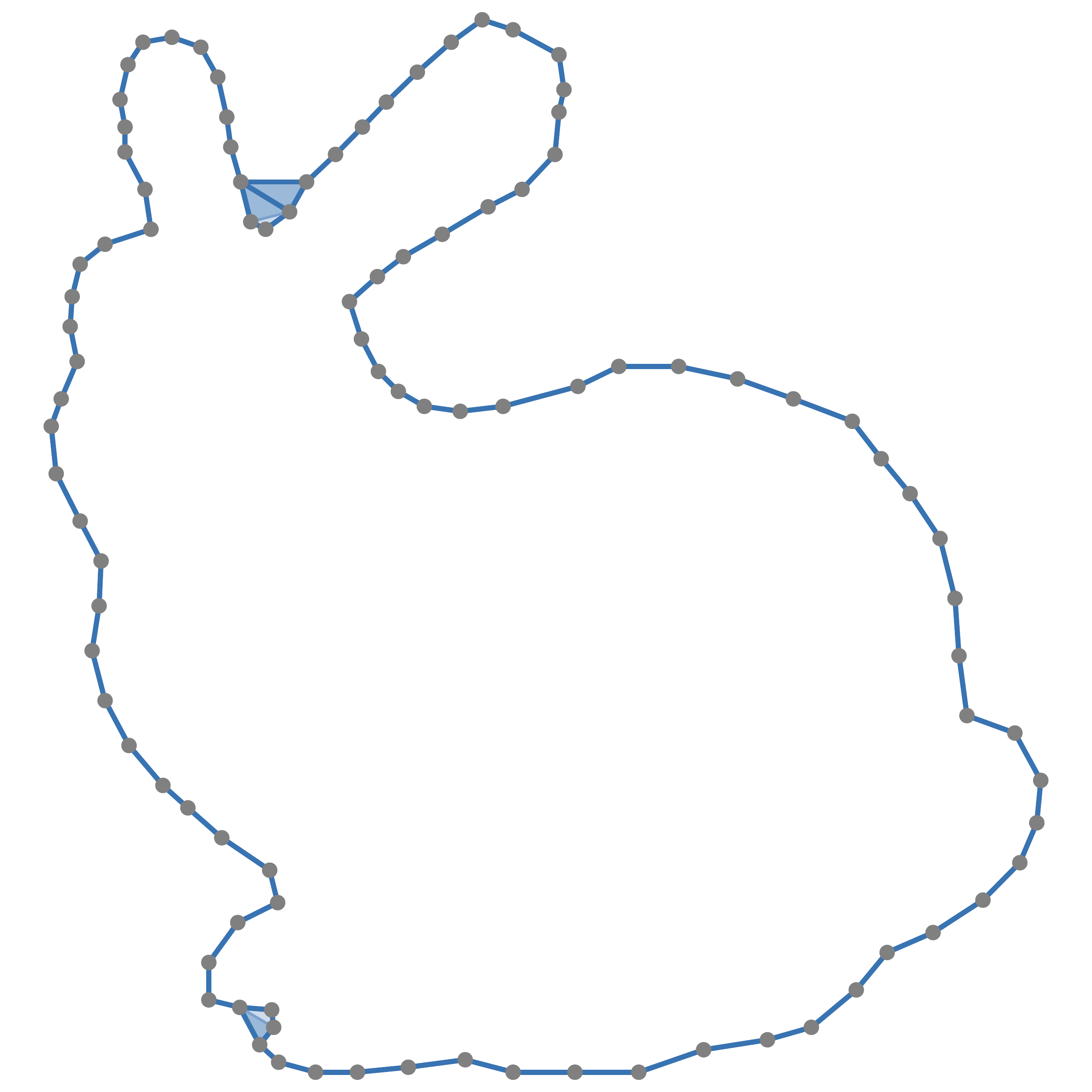}}
  \hfill
    \vcent{\includegraphics[height=.3\textwidth]{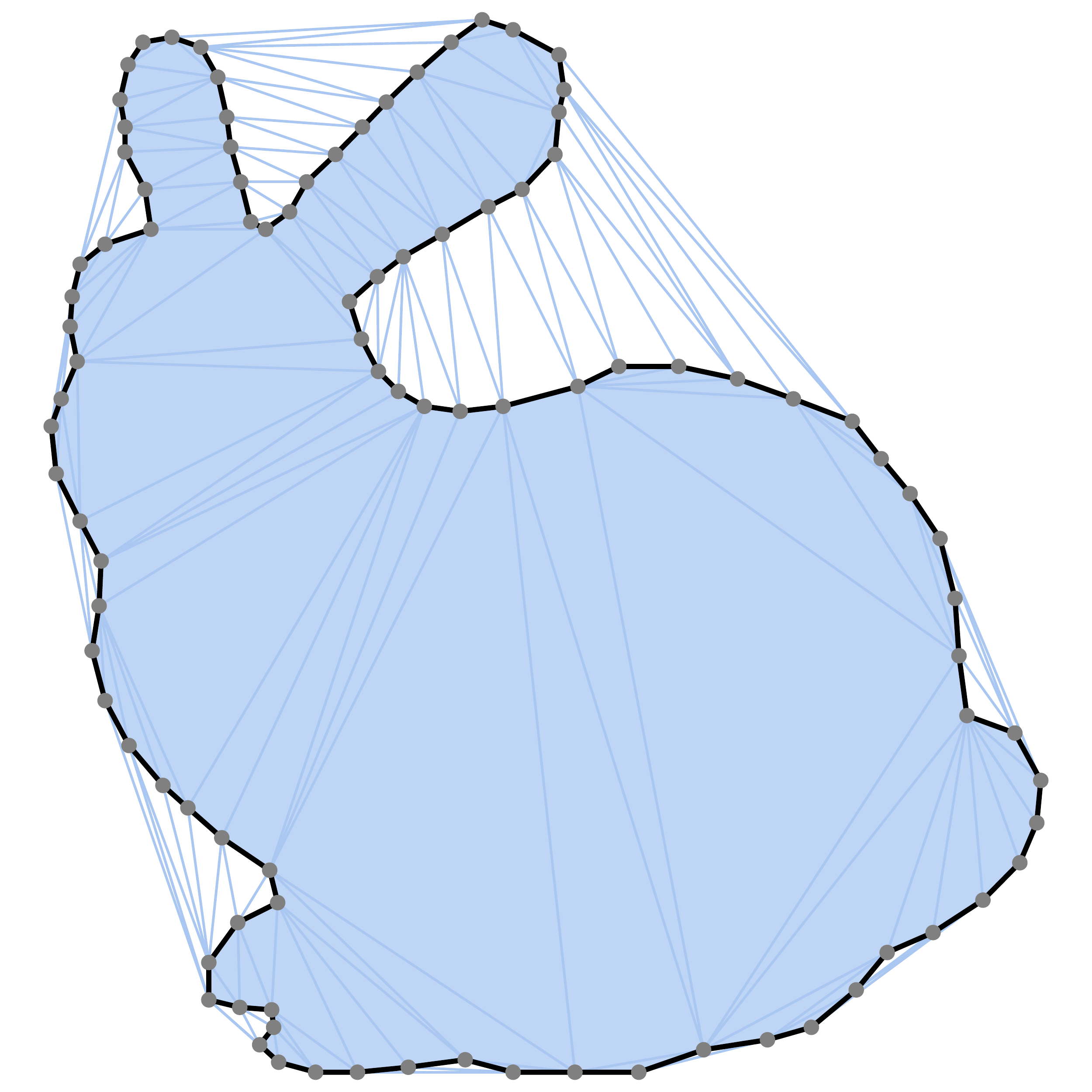}}
  \caption{Left: Delaunay triangulation of a point cloud, with critical simplices highlighted.
  Middle: Wrap complex for a small radius parameter.
  Right: lexicographically minimal cycle for the most persistent feature (black contour), shown together with its bounding chain (shaded blue).}
  \label{fig:2dreduction}
  \end{figure}

As a consequence of our main result, we obtain the following connection between Wrap complexes and persistent homology (see \cref{prelim-pers-hom}) of the Delaunay filtration, specializing \cref{fullyreduced-supported-descending} to the Delaunay radius function. 
Our result states that the minimal cycle homologous to the boundary of a simplex killing homology in the Delaunay filtration -- obtained by a variant of the familiar matrix reduction algorithm for computing persistence -- is supported on the Wrap complex for the birth value of the corresponding feature.
    \begin{corx}
      \label{corollary-b}
      Let $X\subset \mathbb{R}^d$ be a finite subset in general position and $(\sigma,\tau)$ a non-zero~persistence pair of the lexicographically refined Delaunay filtration.
      Let $r=r_X(\sigma)$ and $s=r_X(\tau)$ be the radius of the smallest empty circumsphere of $\sigma$ and of $\tau$, respectively.
      Then the lexicographically minimal cycle of $[\partial\tau]$ in the Delaunay complex $\delo[s]{X}$, given as the column $R_\tau$ of the totally reduced filtration boundary matrix, is supported on the Wrap complex $\wrap[r]{X}$.
    \end{corx}
The totally reduced filtration boundary matrix can be computed using \cref{alg:exhaustive}.
The connection between lexicographically minimal cycles and matrix reduction algorithms used in the context of persistent homology has already been discussed in \cite{MR4517098}.

    For a sufficiently good sample of a compact $d$-submanifold of Euclidean space, the union of closed balls centered at the sample points deformation retracts onto the submanifold by a closest point projection \cite{attali2023tight,MR2383768}.
As the Delaunay complex is naturally homotopy equivalent to the union of closed balls \cite{MR4652781}, this implies that %
    the fundamental class of the manifold is captured directly in the $d$-dimensional persistent homology of the Delaunay filtration through a natural isomorphism.
    Similar observations that are based on the induced maps in persistent homology and with weaker assumptions have also been made before \cite{MR2279866}\cite[Section~11.4]{MR3837127}.
    Moreover, in the preprint \cite[Theorem~9.1]{cohensteiner:hal-02391190} it is argued that the unique non-trivial lexicographically optimal $2$-cycle in the \v{C}ech complex of a sample of a $2$-submanifold of Euclidean space yields a homeomorphic reconstruction of the submanifold.
    Note that the lexicographically optimal cycles considered in \cite{MR4517098,cohensteiner:hal-02391190,MR4598045,vuillamy:tel-03339931} are based on a slightly different total order on simplices, refining the minimum enclosing radius function.
    Relating these particular choices and results to our results and those in \cite{MR3605986} remains an interesting open problem.

    \begin{figure}[t]
        {\includegraphics[height=.3\textwidth]{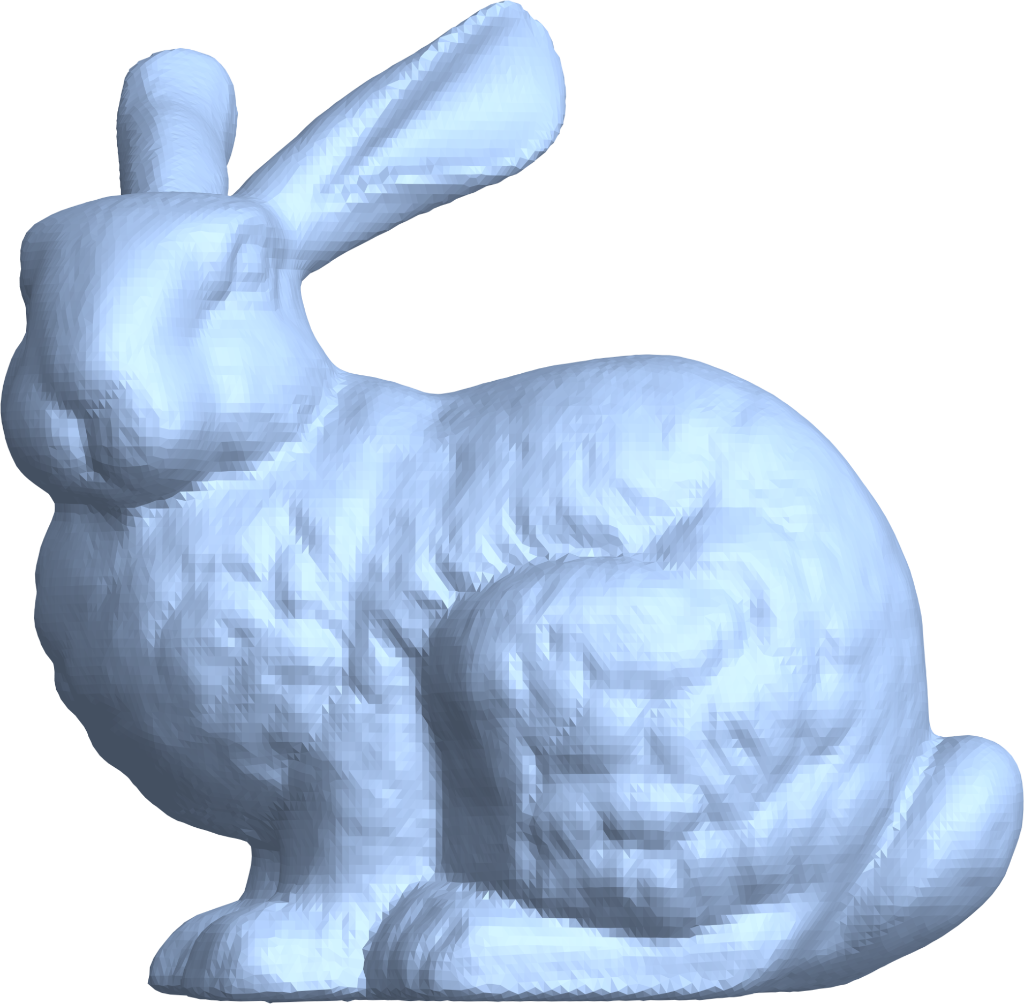}}
      \hfill
        {\includegraphics[height=.5\textwidth]{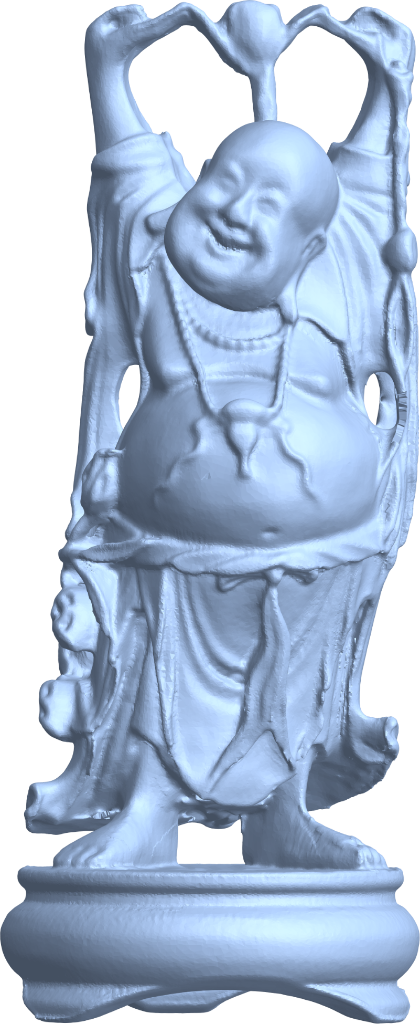}}
      \hfill
        {\includegraphics[height=.3\textwidth]{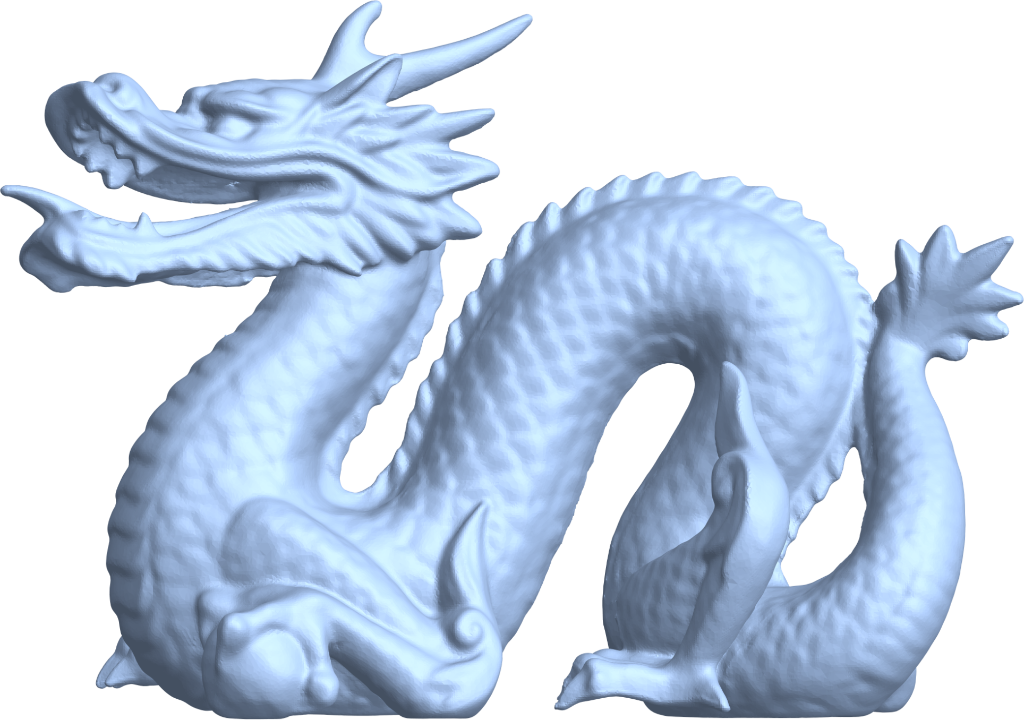}}
      \caption{The lexicographically minimal cycle corresponding to the most persistent feature of the Delaunay filtration for 3D scan point clouds \cite{scans_stanford} yields an accurate reconstruction of the~surface.}
      \label{fig:examples}
      \end{figure}

Together with \cref{corollary-b}, the discussion above suggests the following heuristic for a simple and robust algorithm for shape reconstruction from a point cloud, combining Wrap and lexicographically minimal cycles with persistent homology: Take the most persistent $d$-dimensional feature of the Delaunay filtration, i.e., the interval in the barcode with the largest death/birth ratio.
Intuitively, this feature is born at a small scale and only gets filled in at a large scale.
By \cref{corollary-b}, the corresponding lexicographically minimal cycle is guaranteed to be supported on the Wrap complex for a small scale parameter.
See \cref{fig:examples} for an illustration, which can readily be reproduced using the code provided in \cite{git} by executing \lstinline"docker build -o output github.com/fabian-roll/wrappingcycles" on any machine with Docker installed and configured with sufficient memory (16GB recommended).

\section{Preliminaries}
A finite \emph{simplicial complex $K$} is a collection of finite nonempty sets such that for any set $\sigma\in K$ and any nonempty subset $\rho\subseteq\sigma$ one has~$\rho\in K$.
A set~$\sigma \in K$ is called a \emph{simplex}, and $\dim\sigma=\card\sigma-1$ is its \emph{dimension}.
Moreover, $\rho$ is said to be a \emph{face of $\sigma$} and $\sigma$ a \emph{coface of $\rho$}.
If $\dim\rho=\dim\sigma-1$, then we call $\rho$ a \emph{facet of $\sigma$}, $\sigma$ a \emph{cofacet of $\rho$}, and $(\rho,\sigma)$ a \emph{facet pair}.

\subsection{Persistent homology and apparent pairs}
\label{prelim-pers-hom}
\subparagraph*{Based chain complexes and filtrations}
We assume the reader to be familiar with the basics of homological algebra (see, e.g., \cite{MR755006,MR1269324}).
By a \emph{based chain complex} $(C_*,\chainbasis{*})$ (sometimes also called a \emph{Lefschetz complex} \cite{lefschetz1942algebraic}) we mean a bounded chain complex $C_*=(C_n,\partial)_{n\in \mathbb{N}}$ of finite dimensional vector spaces over a field $\mathbb{F}$ together with a basis $\chainbasis{n}$ for each $C_n$.
Consider the canonical bilinear form $\IP{\cdot}{\cdot}$ on $C_*$ for the given basis $\chainbasis{*}$, i.e., for $a,b\in \chainbasis{*}$ we have $\IP{a}{b}=0$ if $a\neq b$ and $\IP{a}{a}=1$.
Given two basis elements $c\in \chainbasis{n}$ and $e\in \chainbasis{n+1}$ such that $\IP{\partial e}{c}\neq 0$, we call $c$ a \emph{facet} of $e$ and $e$ a \emph{cofacet} of $c$, and we call $(c,e)$ a \emph{facet pair}.
 
For a poset $P$ and any element $p\in P$ we denote by $\downarrow p=\{q\in P \mid q\leq p\}$ the \emph{down set} of $p$.
A \emph{filtration} of $(C_*,\chainbasis{*})$ is a collection of based chain complexes $(C_*^i,\chainbasis{*}^i)_{i\in I}$, where $I$ is a totally ordered indexing set, such that $\chainbasis{*}^i \subseteq \chainbasis{*}$ spans the subcomplex $C_*^i$ of~$C_*$ for all $i \in I$, and $i\leq j$ implies $\chainbasis{*}^i\subseteq \chainbasis{*}^j$.
We call the filtration an \emph{elementwise filtration} if for any $j$ with immediate predecessor $i$ we have that~$\chainbasis{*}^j\setminus \chainbasis{*}^i$ contains exactly one basis element~$\sigma_j$.
Thus, elementwise filtrations of $(C_*,\chainbasis{*})$ correspond bijectively to total orders $<$ on $\chainbasis{*}$ such that the down sets $\downarrow \sigma_j=\{\sigma_i\mid \sigma_i\leq \sigma_j\}$ span subcomplexes.

  Our main example for a based chain complex is the simplicial chain complex $C_*(K)$ of a simplicial complex $K$ with coefficients in a field.
If the vertices of $K$ are totally ordered, then there is a canonical basis of $C_n(K)$ consisting of the $n$-dimensional simplices of $K$ oriented according to the given vertex order, and a simplexwise filtration of $K$ induces a canonical elementwise filtration of $C_*(K)$.

\subparagraph*{Matrix reduction}
For a based chain complex $(C_*,\chainbasis{*}=\sigma_1<\dots<\sigma_l)$ with an elementwise filtration, we often identify an element of $C_*$ with its coordinate vector in $\mathbb{F}^l$.
The \emph{filtration boundary matrix} $D$ of an elementwise filtration is the matrix that represents the boundary map $\partial$ with respect to the total order on the basis elements induced by the filtration.

\begin{figure}[t]
\begin{algorithm}[H]
\caption{Standard matrix reduction}\label{alg:reduction}
\KwIn{$D=\partial$ an $l\times l$ filtration boundary matrix}
\KwResult{$R=D\cdot \VMatrix$ with $R$ reduced and $\VMatrix$ full rank upper triangular}
$R=D$;
$\VMatrix=\op{Id}$\;
\For{$j=1$ \KwTo $l$}
{
    \While{there exists $i< j$ with $\pivot{R_i}=\pivot{R_j}>0$}
    {
        $\mu=-\pivotentry{R_j}/\pivotentry{R_i}$\;
        $R_j=R_j + \mu\cdot R_i$;\
        $\VMatrix_j=\VMatrix_j + \mu\cdot \VMatrix_i$\;
    }
}
  \Return $R,\VMatrix$
\end{algorithm}
\end{figure}

\begin{figure}[t]
\begin{algorithm}[H]
\caption{Exhaustive matrix reduction}\label{alg:exhaustive}
\KwIn{$D=\partial$ an $l\times l$ filtration boundary matrix}
\KwResult{$R=D\cdot \VMatrix$ with $R$ totally reduced and $\VMatrix$ full rank upper triangular}
$R=D$;
$\VMatrix=\op{Id}$\;
\For{$j=1$ \KwTo $l$}
{
    \While{there exist $s, i < j$ with $\pivot{R_i}=s$ and $R_{s,j}\neq 0$}
    {
        $\mu=-R_{s,j}/R_{s,i}$\;
        $R_j=R_j + \mu\cdot R_i$;\
        $\VMatrix_j=\VMatrix_j + \mu\cdot \VMatrix_i$\;
    }
}
  \Return $R,\VMatrix$
\end{algorithm}
\end{figure}

For a matrix $R$, we denote by $R_j$ the $j$th column of $R$, and by $R_{i,j}$ the entry of $R$ in row $i$ and column $j$.
The \emph{pivot} of a column $R_j$, denoted by $\pivot{R_j}$, is the maximal row index~$i$ with $R_{i,j}\neq 0$, taken to be $0$ if all entries are $0$.
Otherwise, the non-zero entry is called the \emph{pivot entry}, denoted by $\pivotentry{R_i}$.
We define $\pivots{R}=\{i \mid i = \pivot{R_j} \neq 0\}$ to be the collection of all non-zero column pivots.
Moreover, we call a column $R_j$ \emph{reduced} if its pivot cannot be decreased by adding a linear combination of the columns $R_i$ with $i<j$, and we call the matrix $R$ reduced if all its columns are reduced.
Finally, we call the matrix $R$ \emph{totally reduced} if for each $i<j$ we have $R_{s,j}=0$, where $s=\pivot{R_i}$.

We call a matrix $\VMatrix$ a \emph{reduction matrix} 
if it is a full rank upper triangular matrix such that $R=D\cdot \VMatrix$ is reduced and $\VMatrix$ is \emph{homogeneous}, meaning that respects the degrees in the chain complex.
%
Any such reduction $R=D\cdot \VMatrix$ of the filtration boundary matrix induces a direct sum decomposition (see, e.g., \cite{MR4298669,MR2854319}) of $C_*$ into elementary chain complexes in the following way:
If $R_j\neq 0$, then we have the summand
\[
\dots \to 0 \to \langle  \VMatrix_j\rangle\stackrel{\partial }{\to} \langle R_j \rangle\to 0\to \dots\ ,
\]
in which case we call $j$ a \emph{death index}, $i=\pivot{R_j}$ a \emph{birth index}, and $(i,j)$ an \emph{index persistence pair}.
If $R_i=0$ and~$i\notin \pivots{R}$, then we have the summand
\[\dots \to 0 \to \langle  \VMatrix_i\rangle\to 0\to \dots\ , \]
in which case we call $i$ an \emph{essential index}.
Moreover, we call the element $\sigma_i$ a \emph{birth}, \emph{death}, or \emph{essential element}, if its index is a birth, death, or essential index.
Similarly, we call a pair of elements $(\sigma_i,\sigma_j)$ a \emph{persistence pair}, if the pair $(i,j)$ is an index persistence pair.
Note that this is independent of the specific reduction of the filtration boundary matrix.
By taking the intersection with the filtration, one obtains elementary filtered chain complexes, in which~$R_j$ is a cycle appearing in the filtration at index $i=\pivot{R_j}$ and becoming a boundary when~$\VMatrix_j$ enters the filtration at index $j$, and in which an essential cycle $\VMatrix_i$ enters the filtration at index~$i$.
Thus, the \emph{barcode} of the \emph{persistent homology} \cite{MR2572029} of the elementwise filtration is given by the collection of intervals $\{[i,j)\mid (i,j) \text{ index persistence pair}\}\cup\{[i,\infty) \mid i \text{ essential index}\}$.

Such a reduction $R=D\cdot \VMatrix$ can be computed by a variant of Gaussian elimination~\cite{MR2389318}, as in \cref{alg:reduction}.
A slight modification is \cref{alg:exhaustive}, which computes a totally reduced filtration boundary matrix, as used in \cref{corollary-b}.
This is also known as \emph{exhaustive reduction}, and appears in various forms in the literature \cite{MR4156247,MR4517098,MR1994939,MR2121296}.

\subparagraph*{Apparent pairs} 
Many optimization schemes have been developed in order to speed up the computation of persistent homology.
One of them is based on \emph{apparent pairs} \cite{MR4298669}, a concept which lies at the interface of persistence and discrete Morse theory.

\begin{definition}
\label{def:app-pairs}
Let $(C_*,\chainbasis{*}=\sigma_1<\dots<\sigma_l)$ be a based chain complex with an elementwise filtration.
We call a pair of basis elements $(\sigma_i,\sigma_j)$ an \emph{apparent pair} if $\sigma_i$ is the maximal facet of~$\sigma_j$ and $\sigma_j$ is the minimal cofacet of $\sigma_i$.
\end{definition}
 In the context of persistence, the interest in apparent pairs stems from the following observation \cite{MR4298669}, immediate from the definitions.

 \begin{lemma}
 \label{app-implies-pers}
 For any apparent pair $(\sigma,\tau)$ of an elementwise filtration, the column of $\tau$ in the filtration boundary matrix is reduced, and $(\sigma,\tau)$ is a persistence pair.
 \end{lemma}

\subsection{Lexicographic optimality}
In this section, we introduce the lexicographic order on chains for a based chain complex $(C_*,\chainbasis{*} = \sigma_1<\dots<\sigma_l)$ with an elementwise filtration, extending the definitions in \cite{MR4517098}.
For any chain $c = \sum_{i}\lambda_i\sigma_i \in C_n$, we define its \emph{support} $\supp{\chainbasis{*}}{c}$ to be the set of basis elements $\sigma_i \in \chainbasis{n}$ with~$\lambda_i\neq 0$.
Note that this is not to be confused with the supporting subcomplex, which also contains the faces of these basis elements.
Given a totally ordered set $(X, \leq)$, we consider the lexicographic order $\preceq$ on the power set $2^X$ given by identifying any subset $A \subseteq X$ with its characteristic function and considering the lexicographic order on the set of characteristic functions.
Explicitly, for $A, B \subseteq X$ we have $A \preceq B$ if and only if $A=B$ or the maximal element of the symmetric difference $(A \setminus B) \cup (B \setminus A)$ is contained in $B$.

\begin{definition}
  The \emph{lexicographic preorder} $\sqsubseteq$ on the collection of chains $C_n$ is given by~$c_1\sqsubseteq c_2$ if and only if $\supp{\chainbasis{*}}{c_1} \preceq \supp{\chainbasis{*}}{c_2}$ in the lexicographic order on subsets of~$\chainbasis{n}$.
We write $\sqsubset$ for the corresponding strict preorder.
\end{definition}

If we consider a simplicial chain complex with coefficients in $\mathbb{Z}/2\mathbb{Z}$, then this preorder is a total order, and it coincides with the one considered in \cite[Definition 2.1]{MR4517098}.

\begin{definition}
  \label{def:lexoptimal}
We call a chain $c\in C_n$ \emph{lexicographically minimal}, or \emph{irreducible}, if there exists no strictly smaller homologous chain $(c + \partial e) \sqsubset c$ in the lexicographic preorder, where~$e\in C_{n+1}$.
Otherwise, we call the chain $c$ \emph{reducible}.
\end{definition}

  It follows from \cref{unique-minimizers} that each homology class of the chain complex $C_*$ has a unique lexicographically minimal representative cycle, regardless of the coefficient field.

\subsection{Discrete Morse theory and apparent pairs}
\label{discrete-morse-apparent-pairs}

Following closely the exposition in~\cite{BauerRoll}, we use the following generalization of a discrete Morse function, originally due to Forman \cite{MR1612391,MR3605986,MR2537376}.
Let $K$ be a simplicial complex.

\begin{definition}
\label{def:gen-morse-fct}
A function $f\colon K\to \mathbb{R}$ is a \emph{generalized discrete Morse function} if
\begin{itemize}
 \item $f$ is monotonic, i.e., for any $\sigma,\tau \in K$ with $\sigma \subseteq \tau$ we have $f(\sigma)\leq f(\tau)$, and
 \item there exists a (unique) partition $\hat{V}$ of $K$ into intervals $[\rho,\phi] = \{ \psi \in K \mid \rho \subseteq \psi \subseteq \phi \}$ in the face poset such that any pair of simplices $\sigma \subseteq \tau$ satisfies $f(\sigma) = f(\tau)$ if and only if $\sigma$ and $\tau$ belong to a common interval in the partition.
\end{itemize}
The collection of \emph{regular} intervals, $[\rho,\phi]$ with $\rho \neq \phi$, is the \emph{discrete gradient} $V$ of $f$ on~$K$, and any singleton interval $[\sigma,\sigma]=\{\sigma\}$, as well as the corresponding simplex $\sigma$, is \emph{critical}.
\end{definition}
If $W$ is another discrete gradient on $K$, then we say that $V$ is a \emph{refinement} of $W$ if each interval in the gradient partition $\hat{W}$ is a disjoint union of intervals in $\hat{V}$.
If the refinement preserves the set of critical simplices, we call it a \emph{regular refinement}.
Moreover, if each regular interval in~$V$ only consists of a pair of simplices, we simply call $f$ a \emph{discrete Morse function}.
We often refer to a discrete gradient without explicit mention of the function $f$, noting that different functions can have the same gradient.

For a monotonic function $g\colon K\to\mathbb{R}$ we write $\sublevel{g}{r}=g^{-1}(-\infty,r]\subseteq K$ for the sublevel set and $\sublevelopen{g}{r}=g^{-1}(-\infty,r)\subseteq K$ for the open sublevel set of $g$ at scale $r\in \mathbb{R}$.

\begin{example}
\label{ex:cech-del}
For a finite subset $X\subset \mathbb{R}^d$ in general position, the \emph{\v{C}ech radius function} $r_\emptyset\colon\cech{X}\to\mathbb{R}$ as well as the \emph{Delaunay radius function} $r_X\colon\del{X}\to\mathbb{R}$, which assign to a simplex the radius of its smallest enclosing sphere and smallest empty circumsphere, respectively, are both generalized discrete Morse functions \cite{MR3605986}.
Moreover, for $r\in\mathbb{R}$ their sublevel sets at scale $r$ are the \emph{\v{C}ech complex} $\cech[r]{X}$ and the \emph{Delaunay complex} $\del[r]{X}$, respectively. Similarly, their open sublevel sets at scale $r$ are the \v{C}ech complex $\cecho[r]{X}$ and the Delaunay complex $\delo[r]{X}$, respectively.
\end{example}

We now explain the connection between discrete Morse theory and apparent pairs (\cref{def:app-pairs}).
Let $f\colon K\to \mathbb{R}$ be a monotonic function and assume that the vertices of $K$ are totally ordered.
The \emph{$f$-lexicographic order} is the total order $\leq_f$ on $K$ given by ordering the~simplices
   by their value under $f$,
   then by dimension,
   then by the lexicographic order induced by the total vertex order.

A persistence pair $(\sigma,\tau)$ of the elementwise filtration determined by the $f$-lexicographic order is a \emph{zero persistence pair} if $f(\sigma)=f(\tau)$.
The collection of apparent pairs of a simplexwise filtration forms a discrete gradient \cite[Lemma 3.5]{MR4298669}, the \emph{apparent pairs gradient}.
There is a further connection between apparent pairs and discrete Morse functions: Let $f\colon K\to \mathbb{R}$ be a generalized discrete Morse function with discrete gradient~$V$.
Consider the regular refinement of $V$ obtained by applying a \emph{minimal vertex refinement} to each interval:
\[\widetilde V=\{(\psi\setminus \{v\},\psi\cup\{v\})\mid \psi\in  [\rho,\phi]\in V,\ v=\min(\phi\setminus\rho) \} .\]
By the following proposition from \cite[Lemma 9]{BauerRoll}, this regular refinement is induced by the apparent pairs gradient for the simplexwise filtration determined by the $f$-lexicographic order.
In particular, the zero persistence pairs of this simplexwise filtration are precisely the zero persistence apparent pairs, and the $V$-critical simplices of $K$ are precisely the simplices that are either essential or contained in a non-zero persistence pair.
\begin{proposition}%
	\label{apparent-equal-gradient}
	 The zero persistence apparent pairs with respect to the $f$-lexicographic order are precisely the gradient pairs of $\widetilde V$.
\end{proposition}

\section{Descending complexes and gradient refinements}
\label{section:descending}

We extend the definition of the \emph{Wrap complex} from \cite{MR3605986} to an arbitrary subset~$\critsubset$ of the set of critical simplices with respect to a discrete gradient $V$ and study its behavior under gradient refinements.
We further extend it to monotonic functions $g\colon K\to \mathbb{R}$ that are \emph{compatible} with $V$: all simplices within an interval $I\in \hat{V}$ have the same function value, $g(I)\in\mathbb{R}$.

The {gradient partition} $\hat{V}$ has a canonical poset structure $\leq_{\hat{V}}$ given by the transitive closure of the relation
$
I\sim J \text{ if and only if there exists a face }\sigma \in I\text{ of a simplex }\tau\in J.
$
The \emph{down set} of a subset $A\subseteq \hat{V}$ is the set of intervals~$\downarrow A = \{I\in \hat{V}\mid \exists J\in A: I\leq_{\hat{V}} J\}$,
and for $r\in \mathbb{R}$ we denote the discrete gradient $V$ restricted to the sublevel set $\sublevel{g}{r}$ by
$
\gradres{V}{r}=\{I\in V\mid g(I)\leq r\}.
$
Note that if $I\leq_{\hat{V}} J$, then $g(I)\leq g(J)$, and hence for a subset $A\subseteq \hat{\gradres{V}{r}}\subseteq \hat{V}$ the down sets with respect to the canonical poset structure on $\hat{V_r}$ and with respect to the canonical poset structure on $\hat{V}$ coincide.

\begin{definition}
  \label{def:desc-cpx}
 For a discrete gradient $V$ on $K$, the \emph{descending complex} is the~subcomplex
 \[\descpxV{V}=\bigcup \downarrow \{\{\sigma\}\mid \sigma\in K \text{ critical}\} \]
 of $K$ given by the union of intervals in the down sets of the critical intervals.
More generally, if $\critsubset$ is a subset of the set of critical simplices, the \emph{descending complex} $\descpxVC{\critsubset}{V}$ is the~subcomplex
 \[\descpxVC{\critsubset}{V}=\bigcup \downarrow \{\{\sigma\}\mid \sigma \in C\}\]
 of $K$.
 Moreover, for a monotonic function $g\colon K\to \mathbb{R}$ that is compatible with $V$, the \emph{descending complex} $\descpx{V}{g}{r}$ at scale $r\in \mathbb{R}$ is the subcomplex
  \[\descpx{V}{g}{r}=\descpxV{\gradres{V}{r}}=\descpxVC{\critsubsetparameter{r}{V}{g}}{V}=\bigcup \downarrow \{\{\sigma\}\mid \sigma\in K \text{ critical, } g(\sigma)\leq r \} \]
  of $\sublevel{g}{r}$, where $\critsubsetparameter{r}{V}{g}=\{\sigma\in K \text{ critical}\mid g(\sigma)\leq r\}$.
If $V$ is the discrete gradient of a generalized discrete Morse function $f$, we simply write $\descpxf{f}{r}$
 for~$\descpx{V}{f}{r}$.
\end{definition}

The descending complex $\descpxV{V}$ in the context of discrete Morse theory is motivated by the concept of a \emph{descending} or \emph{stable manifold} of a critical point from smooth Morse theory, which is central in the original definition of the Wrap complex \cite{MR2038483}.
Note that the descending complex $\descpxf{r_X}{r}\subseteq \del[r]{X}$ of the Delaunay radius function $r_X$ (see \cref{ex:cech-del}) is precisely the \emph{Wrap complex}, $\wrap[r]{X}$, from \cite{MR3605986}.
It has been shown that a variant of the Wrap complex can be used to topologically reconstruct a submanifold by choosing a suitable subset of critical simplices \cite{MR2460367,MR2485260,DBLP:journals/cgf/Sadri09}, which also motivates our definition of $\descpxVC{\critsubset}{V}$.

We now study the behavior of the descending complex under gradient refinements.
\begin{proposition}%
  \label{refinement-implies-nested-descendingcpxV}%
  Let $V$ be a discrete gradient on $K$, let $\critsubset$ be a subset of $V$-critical simplices, and let $W$ be a refinement of $V$.
  Then the descending complex $\descpxVC{C}{W}$ is a subcomplex of the descending complex $\descpxVC{C}{V}$.
  \end{proposition}
  \begin{proof}
    Note first that every $V$-critical simplex is also $W$-critical, as $W$ is a refinement of $V$.
    By the same reason, there exists a set map $\varphi\colon \hat{W}\to \hat{V}$ between the gradient partitions such that for every $B\in \hat{W}$ we have~$B\subseteq \varphi(B)$.
    It follows straightforwardly from the definition of the canonical poset structures on $\hat{W}$ and~$\hat{V}$ that $\varphi$ is a poset map.
  Thus, for every~$W$-critical simplex $\sigma\in C$ and interval $A\in \hat{W}$ with $A\leq_{\hat{W}} \{\sigma\}$, we have $\varphi(A)\leq_{\hat{V}} \varphi(\{\sigma\})=\{\sigma\}\in \hat{V}$, as $\sigma$ is also $V$-critical.
   It now follows directly from the construction of the descending complexes, that $A\subseteq \varphi(A)\subseteq\descpxVC{C}{V}$ and $\descpxVC{C}{W}\subseteq\descpxVC{C}{V}$.
  \end{proof}

  \begin{remark}
If $L$ is a subcomplex of $K$ and the complement $K\setminus L$ is the disjoint union of regular intervals in $V$, then $V$ induces a collapse $K\searrow L$ \cite[Theorem 2.2]{MR3605986}.
It follows directly from this and the construction of $\descpxV{V}$, that $\descpxV{V}$ is the smallest subcomplex of~$K$ such that $V$ induces a collapse $K\searrow \descpxV{V}$.
Moreover, if $W$ is a regular refinement of $V$, implying $\descpxV{W}\subseteq \descpxV{V}$ by \cref{refinement-implies-nested-descendingcpxV},
the complement $\descpxV{V}\setminus \descpxV{W}$ is the disjoint union of regular intervals in $W$.
Similar to before, it follows that $W$ induces a collapse $\descpxV{V}\searrow \descpxV{W}$.
In particular, the inclusion $\descpxV{W}\hookrightarrow \descpxV{V}$ is a homotopy equivalence (see~\cref{fig:refinement-heq}).
      
   \begin{figure}[t]
    \begin{minipage}[c]{0.5\linewidth}
    \centering
      \vcent{\includegraphics[width=.65\textwidth]{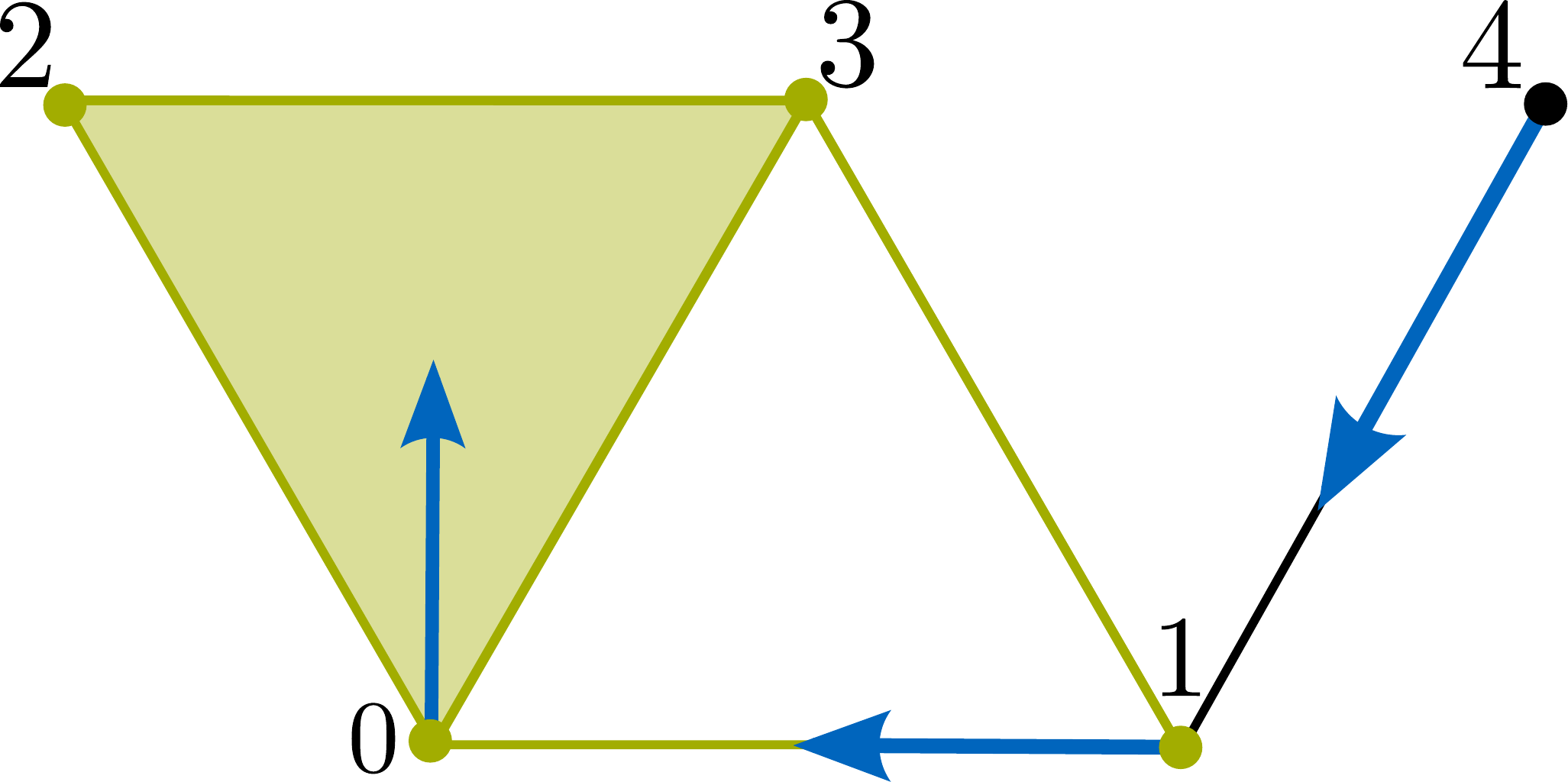}}
    \end{minipage}
    \hfill
    \begin{minipage}[c]{0.5\linewidth}
    \centering
      \vcent{\includegraphics[width=.65\textwidth]{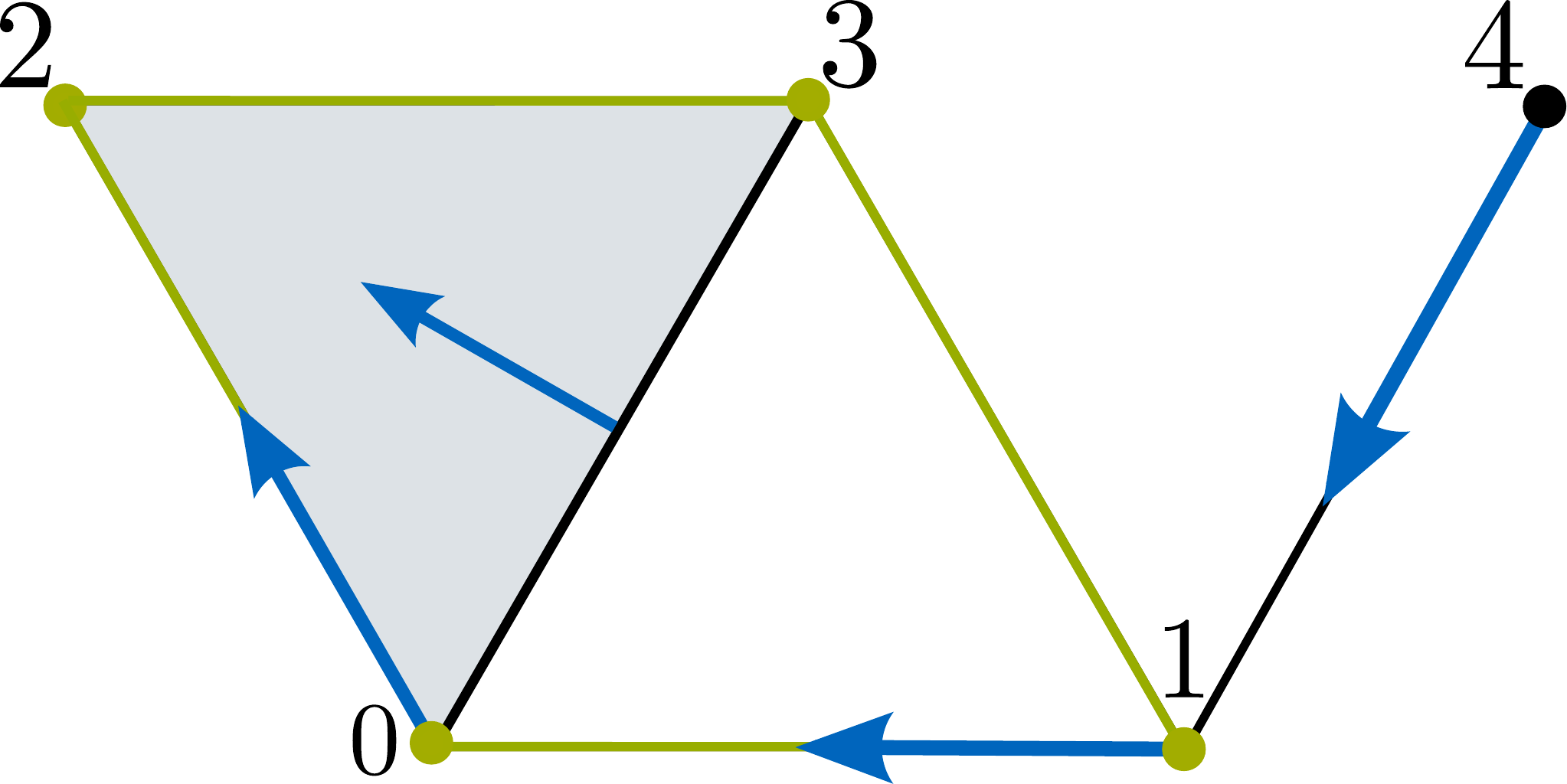}}
    \end{minipage}
    \caption{Left: Generalized discrete gradient (blue) with corresponding descending complex (green).
    Right: lexicographic gradient refinement (blue) with corresponding descending complex~(green).}
    \label{fig:refinement-heq}
    \end{figure}
  \end{remark}

\section{Algebraic Morse theory and persistence}
\label{algmorse-persistence}
We saw that the apparent pairs are closely related to persistent homology and discrete Morse theory.
In this section, we show how all the persistence pairs are related to algebraic Morse theory \cite{MR2171225, MR2488864, kozlov}, also called algebraic discrete Morse~theory.
We also show how this approach connects to lexicographically minimal cycles. %
Let $(C_*,\chainbasis{*})$ be a based chain~complex.
\begin{definition}
\label{def:alg-morse-fct}
A function $f\colon \chainbasis{*}\to \mathbb{R}$ is an \emph{algebraic Morse function} if
\begin{itemize}
 \item $f$ is monotonic, i.e., for any facet $\sigma\in\chainbasis{n}$ of $\tau\in\chainbasis{n+1}$ we have $f(\sigma)\leq f(\tau)$, and
 \item there exists a (unique) disjoint collection $V$ of facet pairs such that every facet pair $(\mu,\eta)$ satisfies $f(\mu)=f(\eta)$ if and only if $(\mu,\eta) \in V$.
\end{itemize}
We call $V$ the \emph{algebraic gradient} of $f$ on $\chainbasis{*}$, and a basis element \emph{critical} if it is not contained in any pair of $V$. Moreover, for $(\sigma,\tau) \in V$ we call $\sigma$ a \emph{gradient facet} and $\tau$ a \emph{gradient~cofacet}.
\end{definition}
We often refer to an algebraic gradient without explicit mention of the associated function.

\begin{remark}
\label{top-to-alg-gradient}
The definitions of algebraic Morse function and algebraic gradient generalize those from discrete Morse theory.
Let $f\colon K\to\mathbb{R}$ be a discrete Morse function with discrete gradient $V$.
Recall that the simplicial chain complex $C_*(K)$ has a basis $\chainbasis{*}$ consisting of the simplices of $K$ with some chosen orientation.
We can now interpret $f$ as an algebraic Morse function on this basis $\chainbasis{*}$ and the discrete gradient $V$ as an algebraic gradient.
\end{remark}

\subsection{Gradient pairs from persistence pairs}
\label{gradientpairs-persistentpairs}

We now explain how all persistence pairs determine an algebraic gradient that relates to discrete Morse theory through apparent pairs (\cref{app-implies-pers,apparent-equal-gradient}).
This establishes the framework for a key step in our proof of \cref{theorem-a}.
Let $(C_*,\chainbasis{*}=\sigma_1<\dots<\sigma_l)$ be a based chain complex with an elementwise filtration, and let $R=D\cdot \VMatrix$ be a reduction of the filtration boundary matrix.
For any chain $c\in C_n$, we denote by $\chainpivot{\chainbasis{*}}{c}=\max\supp{\chainbasis{*}}{c}$ the maximal basis element in the basis representation of $c$ with respect to $\chainbasis{*}$.
If $v$ is the coordinate vector in $\mathbb{F}^l$ representing~$c$, we also write $\chainpivot{\chainbasis{*}}{v}$ for $\chainpivot{\chainbasis{*}}{c}=\sigma_{\pivot{v}}$.

The direct sum decomposition of filtered chain complexes explained in \cref{prelim-pers-hom} yields a straightforward interpretation of persistence pairs as an algebraic gradient, which we discuss in \cref{appendix:gradientpairs-persistentpairs}. However, this gradient is not suitable for our purposes as it neither relates directly to apparent pairs nor lexicographically optimal cycles. Hence, we consider an alternative approach that uses the distinctness of non-zero pivot elements in the reduced matrix $R=D\cdot \VMatrix$.
To this end, we equip the chain complex $C_*$ with the new ordered basis $\Omega_{*}=\tau_1<\dots <\tau _l$ given by
\[
\tau_i=
\begin{dcases*}
\sigma_i & if $i$ is a birth  or essential index,\\
S_i & if $i$ is a death index.
\end{dcases*}
\]
We call this basis $\Omega_*$ the \emph{reduction basis}.
Note that with respect to the original basis we have $\chainpivot{\chainbasis{*}}{\tau_i}=\sigma_i$ for all $i$.
Moreover, note that for every death index $j$ and $R_j=D\cdot \VMatrix_j$ we have $\chainpivot{\Omega_{*}}{R_j}=\chainpivot{\chainbasis{*}}{R_j}$.
By pairing the death columns $\VMatrix_j$ with the pivot elements $\chainpivot{\Omega_{*}}{R_j}$ of their boundaries $R_j=D\cdot \VMatrix_j$, we obtain a set of disjoint pairs, which we call the \emph{reduction gradient} %
of $S$:
\[
\{(\chainpivot{\Omega_{*}}{R_j},\VMatrix_j)\mid\ j\text{ is a death index}\}.
\]

\begin{proposition}
\label{pivot-gradient-is-gradient}
 The reduction gradient is an algebraic gradient on $\Omega_{*}$.
\end{proposition}
\begin{proof}
 Consider the function $f\colon \Omega_{*}\to\mathbb{N}$ with values $f(\chainpivot{\Omega_{*}}{R_j})=f(\VMatrix_j)={\pivot{R_j}}$ for every death index $j$ and $f(\tau_i)=i$ for every essential index $i$.
Note that, as $S$ is a full rank upper triangular matrix, the ordered basis~$\Omega_*= \tau_1<\dots <\tau _l$ is compatible with the given elementwise filtration, in the sense that $\downarrow \tau_j$ induces the same subcomplex of $C_*$ as $\downarrow \sigma_j$ for every $j$.
In particular, whenever $\tau_i$ is a facet of $\tau_j$, we have $i<j$.
The function $f$ assigns to any basis element of the form $\tau_j=\VMatrix_j$, for~$j$ a death index, the index $i=\pivot{R_j}$ of its maximal facet $\chainpivot{\Omega_{*}}{R_j}=\chainpivot{\chainbasis{*}}{R_j}=\tau_i\in \Omega_*$, and to all other basis elements $\tau_i$ their index $i$.
This implies that $f$ is monotonic with algebraic gradient the reduction~gradient.
\end{proof}

The reduction gradient is closely related to apparent pairs (\cref{def:app-pairs}) in the following sense.
For any apparent pair $(\sigma_i,\sigma_j)$ of the elementwise filtration, the column of~$\sigma_j$ in the filtration boundary matrix is reduced (\cref{app-implies-pers}).
Therefore, we are led to define a reduction matrix $\VMatrix$ to be \emph{apparent pairs compatible} if the column $\VMatrix_j$ contains only one non-zero entry $\VMatrix_{j,j}=1$ for every apparent pair $(\sigma_i,\sigma_j)$.
Note that the collection of apparent pairs of the elementwise filtration forms an algebraic gradient on $\chainbasis{*}$ \cites[Lemma 2.2]{lampret2020chain}[Lemma 3.5]{MR4298669}\ that we call the \emph{apparent pairs gradient}.
The following is a direct consequence of the definitions.

\begin{lemma}
  \label{apparent-pairs-contained-reduction-gradient}
If the reduction matrix $S$ is apparent pairs compatible, then
the apparent pairs gradient of the elementwise filtration on $\chainbasis{*}$ is a subset of the reduction gradient of $S$ on $\Omega_{*}$.
\end{lemma}

\begin{remark}
\label{reduction-homogeneous}
Both of the reduction algorithms in \cref{prelim-pers-hom} compute a reduction $R=D\cdot \VMatrix$ of the filtration boundary matrix, noting explicitly that $S$ is homogeneous.
Furthermore, \cref{alg:reduction} computes a reduction matrix $\VMatrix$ that is also apparent pairs compatible.
\end{remark}

\subsection{The flow of an algebraic gradient}
\label{section-alg-flow}
We now introduce the flow determined by an algebraic gradient, study its behavior under gradient containment, and analyze in the subsequent sections its relation to lexicographically minimal cycles (\cref{invpivotflow-equ-lexminsimplicial}), as well as the descending complex (\cref{flowinf-supported-descending}).
While the remainder of the paper is mainly focused on cycles, in this section we present results that hold more generally for chains, and which may be of independent interest. Moreover, in \cref{appendix-algflow-matrixred} we demonstrate how the algebraic flow on a cycle can be interpreted as a variant of Gaussian elimination, tying it closely to the exhaustive reduction (\cref{alg:exhaustive}).

Let $(C_*,\chainbasis{*})$ be a based chain complex and $V$ an algebraic gradient on $\chainbasis{*}$.
The following definition, originally for discrete gradients \cite{MR1612391}, carries over naturally to the algebraic setting.
\begin{definition}%
\label{def-alg-flow}
The \emph{flow} $\Phi\colon C_*\to C_*$ determined by $V$ is the chain map given by
\[\Phi(c)=c+\partial \Flow(c)+\Flow(\partial c),\]
where $\Flow\colon C_*\to C_{*+1}$ is the unique linear map defined on the basis elements $\sigma \in \chainbasis{*}$ as 
\[
\Flow(\sigma)=
\begin{dcases*}
-\textstyle\frac1{\IP{\partial \tau}{\sigma}} \cdot \tau & if $\sigma$ is contained in a pair $(\sigma,\tau)\in V$, \\
0 & otherwise.
\end{dcases*}
\]
\end{definition}
  Note that, by construction, the map $\Flow$ is a chain homotopy between the identity and the flow $\Phi$.
In particular, if $c$ is a cycle, then the flow reduces to $\Phi(c)=c+\partial \Flow(c)$ and therefore acts on each homology class of the chain complex by a change of representative cycle.
  Moreover, if $f\colon \chainbasis{*}\to \mathbb{R}$ is an algebraic Morse function, then, by construction, the associated flow $\Phi$ acts for any $r \in \mathbb{R}$ on the subcomplex of $C_*$ spanned by the sublevel set $\sublevel{f}{r}=f^{-1}(-\infty,r]$.

Forman \cite{MR1612391} proved that the sequence $(\Phi^n)_n$ stabilizes for the cellular chain complex of a finite CW-complex.
This generalizes to our setting of arbitrary finite chain complexes, and we denote the \emph{stabilized flow} by $\Phi^\infty=\Phi^{r}$, where $r$ is a large enough natural number:

\begin{proposition}
 There exists an $r\in \mathbb{N}$ such that for all $s\geq r$ we have $\Phi^s=\Phi^{r}$.
\end{proposition}

The proof of this proposition uses the following lemma, which is a straightforward generalization of the corresponding statement in \cite{MR1612391}, and is of independent interest to us.
\begin{lemma}
  \label{stabilized-critical}
   Let $\sigma\in \chainbasis{n}$ be critical.
  Then for all $r\in \mathbb{N}$ the iterated flow $\Phi^{r+1}(\sigma)$ is given~by 
  \[\Phi^{r+1}(\sigma)=\sigma+w+\Phi(w)+\dots+\Phi^{r}(w),\]
   where $w=\Flow(\partial\sigma)$.
  Moreover, we have $\Phi^s(w)\in \op{im} \Flow$ for all $s\in \mathbb{N}$, and $\Flow(\Phi^{r+1}(\sigma))=0$.
  \end{lemma}

Based on the concept of a discrete flow, Forman \cite{MR1612391} considered the subcomplex of the cellular chain complex given by the flow invariant chains and proved the following proposition, which generalizes to our setting.
Consider the subcomplex of $\Phi$-invariant chains,
\[C_*^\Phi=\{c\in C_*\mid \Phi(c)=c\}.\]

\begin{proposition}%
\label{characterization-inv-chains}
The $\Phi$-invariant chains are spanned by the image of the critical basis elements under the stabilized flow $C_n^\Phi=\op{span}\{\Phi^\infty(\sigma)\mid \sigma\in \chainbasis{n} \text{ critical}\}$.
\end{proposition}

We now relate the flow invariant chains of an algebraic gradient to those of its subgradients.
A proof of the following statement can be found in \cref{appendix-alg-flow}.
\begin{proposition}
\label{comp-inv-chains-ref}
Let $(C_*,\chainbasis{*})$ be a based chain complex and $W\subseteq V$ two algebraic gradients on $\chainbasis{*}$.
Consider the flows $\Psi,\Phi\colon C_*\to C_*$ determined by $W$ and $V$, respectively.
Then any~$\Phi$-invariant chain is also $\Psi$-invariant, i.e., we have $C_*^\Phi\subseteq C_*^\Psi$.
\end{proposition}

Note that the flow $\Phi$ can be written as a sum of flows
$
\Phi = \sum_{(a,b)\in V}\Phi^{(a,b)}-(\card{V}-1)\cdot\op{id},
$
where $\Phi^{(a,b)}$ is the flow determined by the algebraic gradient $\{(a,b)\}$ on $\chainbasis{*}$.
Together with \cref{comp-inv-chains-ref}, this proves the following.
\begin{corollary}
\label{inv-iff-vecflow-inv}
Let $V$ be an algebraic gradient on $\chainbasis{*}$ with associated flow $\Phi\colon C_*\to C_*$.
Then a chain is $\Phi$-invariant if and only if it is $\Phi^{(a,b)}$-invariant for every pair~$(a,b)\in V$.
\end{corollary}

For a based chain complex $(C_*,\chainbasis{*})$ and algebraic gradient $V$ on $\chainbasis{*}$, we denote by 
$
  \gradbound{n}{V}=\{\partial b\mid \exists\ (a,b)\in V \text{ with }a\in \chainbasis{n}\}
$
the set of gradient cofacet boundaries in degree $n$.
We say that two cycles $z,z'\in Z_n$ are \emph{$V$-homologous} if there exists an element $\partial e\in \spann{\gradbound{n}{V}}$ such that $z-z'=\partial e$.
Observe that for any cycle $z\in Z_n$, the cycle $\Phi(z)=z+\partial \Flow(z)$ is~$V$-homologous to $z$.
A proof of the following statement can be found in \cref{appendix-alg-flow}.
\begin{proposition}
\label{phiinv-equ-nogradientfacets}
Let $V$ be an algebraic gradient on $\chainbasis{*}$ with associated flow $\Phi\colon C_*\to C_*$.
Then a cycle $z\in Z_n$ is $\Phi$-invariant if and only if it contains no gradient facets of~$V$.
Moreover, if $z'\in Z_n$ is any $\Phi$-invariant cycle that is $V$-homologous to $z$, then~$z=z'$.
\end{proposition}

 Now let $(C_*,\chainbasis{*}=\sigma_1<\dots<\sigma_l)$ be a based chain complex with an elementwise filtration, and let $R=D\cdot \VMatrix$ be a reduction of the filtration boundary matrix. %
 In order to relate the flow determined by the reduction gradient to the flow determined by the zero persistence apparent pairs gradient, and therefore to discrete Morse theory (\cref{apparent-equal-gradient,flowinf-supported-descending}), we call a reduction matrix $\VMatrix$ \emph{death-compatible} %
 if for every death index $j$ and non-zero entry $\VMatrix_{i,j}\neq 0$ we also have that $i$ is a death index.
 
 \begin{remark}
  \label{reduction-deathcomp}
  Both algorithms in \cref{prelim-pers-hom} compute a death-compatible reduction matrix~$\VMatrix$.
 \end{remark}
 
 The following is a direct consequence of \cref{apparent-pairs-contained-reduction-gradient} and the definitions.

 \begin{lemma}
  \label{apparentpairs-flow-differentbasis}
If the reduction matrix $S$ is apparent pairs and death-compatible, then
the flows determined by the apparent pairs gradient of the elementwise filtration as an algebraic gradient on $\chainbasis{*}$ and as an algebraic gradient on $\Omega_{*}$, respectively, coincide.
\end{lemma}

\subsection{Relating the algebraic flow and lexicographically minimal cycles}
\label{algflow-lexoptcyc}
We now relate the flow invariant cycles determined by a reduction gradient to lexicographically minimal homologous cycles.
Let $(C_*,\chainbasis{*}=\sigma_1<\dots<\sigma_l)$ be a based chain complex with an elementwise filtration, and let $R=D\cdot \VMatrix$ be a reduction of the filtration boundary matrix. %
The following result \cite[Lemmas 3.3 and 3.5]{MR4517098} provides an equivalent condition for minimality.
\begin{proposition}%
\label{lexmin-iff-nobirthspx}
 A cycle is lexicographically minimal with respect to the elementwise filtration if and only if its support contains only death elements and essential elements.
\end{proposition}

We provide another characterization in terms of the algebraic flow determined by the reduction gradient of $\VMatrix$ (defined on the reduction basis $\Omega_*$), which we call the \emph{reduction flow}.

\begin{proposition}
\label{invpivotflow-equ-lexminsimplicial}
 Let $(C_*,\chainbasis{*}=\sigma_1<\dots<\sigma_l)$ be a based chain complex with an elementwise filtration, and 
 let $R=D\cdot \VMatrix$ be a reduction of the filtration boundary matrix by a death-compatible reduction matrix $\VMatrix$.
 Then for a cycle $z\in Z_n$ the following are equivalent:
 \begin{itemize}
  \item $z$ is lexicographically minimal with respect to the ordered basis $\chainbasis{*}$;
  \item $z$ is invariant under the reduction flow.%
 \end{itemize}
\end{proposition}

We defer the proofs to \cref{appendix-flow-equ-lex}.

\section{Relating algebraic reduction gradients and discrete gradients}
\label{relating-gradients}
We are now ready to relate our results for the algebraic flow determined by the reduction gradient to discrete Morse theory.
In particular, we show that for a generalized discrete Morse function $f\colon K\to\mathbb{R}$, the lexicographically minimal cycles with respect to the $f$-lexicographic order on the simplices are supported on the descending complexes of $f$. In \cref{appendix-reduction-chains}, we show that the reduction chains, given as the columns of a suitable reduction matrix, are also supported on the descending complexes of $f$.
In particular, this is true of the columns of the totally reduced filtration boundary matrix, considered as cycles in the sublevel set corresponding to the pivot index.
This summarizes the relations between reductions of the filtration boundary matrix and the descending complexes of~$f$.
Recall that a discrete gradient $V$ on $K$ gives rise to a flow $\Phi\colon C_*(K)\to C_*(K)$.
\begin{lemma}
\label{flow-acts-wrap}
The flow $\Phi$ restricts to a chain map on the descending complex $\descpxV{V}\subseteq K$, i.e., for $c\in C_*(\descpxV{V})$ we also have $\Phi(c)\in C_*(\descpxV{V})$.
\end{lemma}
\begin{proof}
Recall from \cref{def-alg-flow} that $\Phi$ is given by $\Phi(c)=c+\partial \Flow(c)+\Flow(\partial c)$, where $\Flow\colon C_*(K)\to C_{*+1}(K)$ is the linear map with $\Flow(\sigma)=-\IP{\partial \tau}{\sigma}^{-1} \cdot \tau$ if $(\sigma,\tau)\in V$ and~$0$ on all other simplices.
Thus, if $\eta\in \descpxV{V}$ is any simplex, then $\partial\eta$ and $\Flow(\eta)$ are both contained in $C_*(\descpxV{V})$, by definition of the descending complex.
Therefore, if $c\in C_*(\descpxV{V})$ is any chain, then the chains $\partial c, \Flow(c)$, and~$\Flow(\partial c)$ are also contained in $C_*(\descpxV{V})$.
This shows that the chain $\Phi(c)=c+\partial \Flow(c)+\Flow(\partial c)$ is contained in $C_*(\descpxV{V})$, proving the claim.
\end{proof}

\goodbreak

The $\Phi$-invariant chains are supported on the descending complex, as illustrated in~\cref{fig:cycleandflow}.

\begin{proposition}
\label{flowinf-supported-descending}
The $\Phi$-invariant chains of $C_*(K)$ are supported on the descending complex~$\descpxV{V}$, i.e., we have $C_*^\Phi(K)\subseteq C_*(\descpxV{V})$.
\end{proposition}

\begin{proof}
Let $c\in C_*^\Phi(K)$ be any $\Phi$-invariant chain; we show that $c$ is contained in $C_*(\descpxV{V})$.
By \cref{characterization-inv-chains}, we can assume without loss of generality that $c$ is of the form $\Phi^\infty(\sigma)$ for a critical simplex $\sigma$.
By definition, $\Phi^\infty=\Phi^{r}$ for a large enough~$r$, and by \cref{stabilized-critical}, $\Phi^{r}(\sigma)$ is given by $\Phi^{r}(\sigma)=\sigma+w+\Phi(w)+\dots+\Phi^{r-1}(w)$, where $w=\Flow(\partial\sigma)$.
 It follows directly from the definition of the descending complex (\cref{def:desc-cpx}) that $\sigma, \partial\sigma$, and $w=\Flow(\partial \sigma)$ are contained in $C_*(\descpxV{V})$.
By \cref{flow-acts-wrap}, we have $\Phi^k(w)\in C_*(\descpxV{V})$ for every~$k$, and therefore $c=\Phi^{r}(\sigma)=\sigma+w+\Phi(w)+\dots+\Phi^{r-1}(w)\in C_*(\descpxV{V})$, proving the claim.
\end{proof}

\begin{figure}[t]
  \begin{minipage}[c]{0.5\linewidth}
 \centering
   \vcent{\includegraphics[width=.7\textwidth]{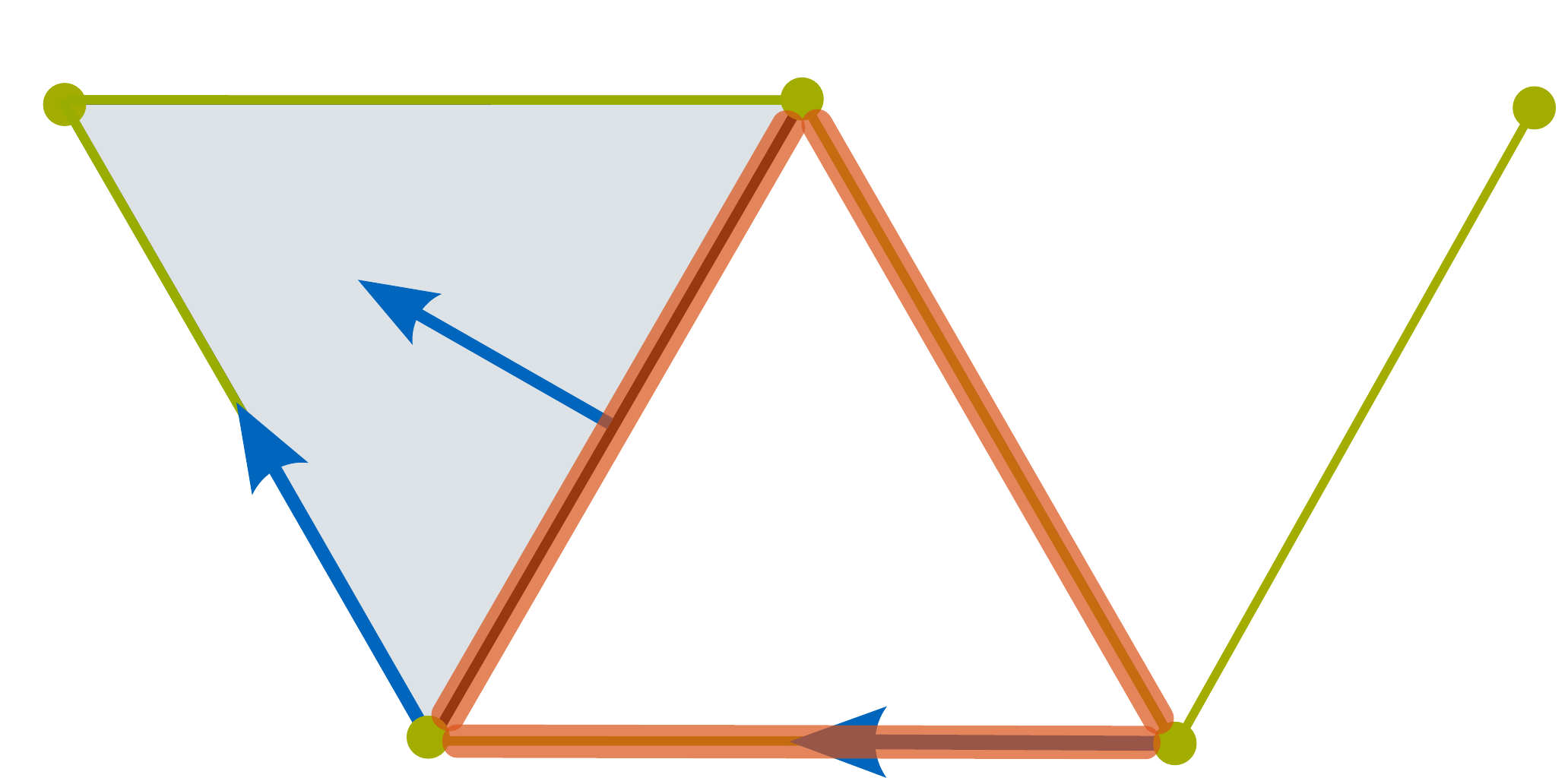}}
 \end{minipage}
 \hfill
 \begin{minipage}[c]{0.5\linewidth}
 \centering
   \vcent{\includegraphics[width=.7\textwidth]{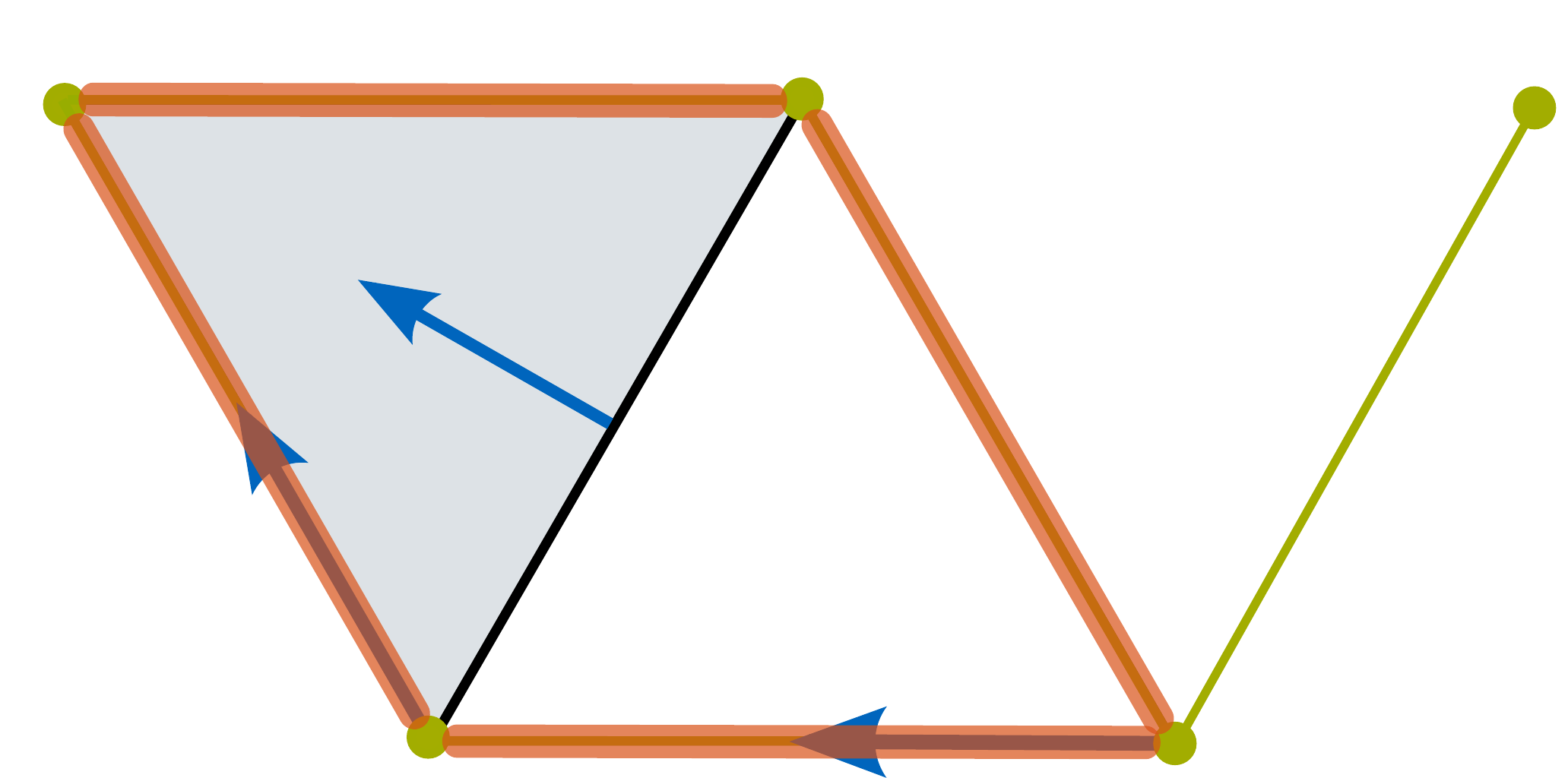}}
 \end{minipage}
 \caption{Discrete gradient (blue) with corresponding descending complex (green).
 Left: Cycle $c$ (red).
 Right: Stabilized cycle $\Phi(c)=\Phi^\infty(c)$ (red), supported on the descending complex (green).}
 \label{fig:cycleandflow}
 \end{figure}

The following, together with \cref{ex:cech-del}, directly implies \cref{theorem-a}.
\begin{theorem}
\label{inv-chain-supp-on-desc-cpx}
Let $f$ be a generalized discrete Morse function, let $r \in \mathbb R$,
and let $h \in~H_*(\sublevel{f}{r})$ be a homology class of the sublevel set $\sublevel{f}{r}$.
Then the lexicographically minimal cycle of $h$, with respect to the $f$-lexicographic order, is supported on the descending complex~$\descpxf{f}{r}$.
\end{theorem}
\begin{proof}
  Let $V$ be the discrete gradient of $f$, and let $W$ be the zero persistence apparent pairs gradient induced by the $f$-lexicographic order, which is a regular refinement of $V$ by \cref{apparent-equal-gradient}.
  Note that $f$ is compatible with $W$.

  Recall from \cref{reduction-homogeneous,reduction-deathcomp} that \cref{alg:reduction} computes a reduction $R=D\cdot \VMatrix$ of the filtration boundary matrix, corresponding to the simplexwise filtration of $\sublevel{f}{r}$ induced by the $f$-lexicographic order, such that the reduction matrix $\VMatrix$ is apparent pairs compatible and also death-compatible.
Consider the corresponding reduction gradient on the corresponding reduction basis $\Omega_{*}$ with associated reduction flow $\Psi\colon C_*(\sublevel{f}{r})\to C_*(\sublevel{f}{r})$.
 
 The lexicographically minimal cycle $z\in Z_*(\sublevel{f}{r})$ of $h$ is a $\Psi$-invariant cycle according to \cref{invpivotflow-equ-lexminsimplicial}.
 We show that $z$ is contained in $C_*(\descpxf{f}{r})$, which proves the claim:
As~$\VMatrix$ is apparent pairs compatible, it follows from \cref{apparent-pairs-contained-reduction-gradient,comp-inv-chains-ref} that $z$ is also invariant under the algebraic flow determined by the zero persistence apparent pairs gradient $W_r$ on $\Omega_{*}$.
\cref{apparentpairs-flow-differentbasis} implies that this flow coincides with the algebraic flow $\Phi\colon C_*(\sublevel{f}{r})\to C_*(\sublevel{f}{r})$ determined by the zero persistence apparent pairs gradient~$W_r$ on the standard basis given by the simplices of $\sublevel{f}{r}$.
It now follows from \cref{flowinf-supported-descending}, the definition of descending complex (\cref{def:desc-cpx}), and 
\cref{refinement-implies-nested-descendingcpxV}
that $z$ is contained in 
\[C_*^\Phi(\sublevel{f}{r})\subseteq C_*(\descpxV{W_r})=C_*(\descpx{W}{f}{r})\subseteq C_*(\descpx{V}{f}{r})=C_*(\descpxf{f}{r}).\qedhere \]
\end{proof}

Let $D$ be the filtration boundary matrix of the simplexwise filtration of $K$ induced by the $f$-lexicographic order, and let $R=D\cdot \VMatrix$ be a totally reduced reduction of $D$.
The following, together with \cref{ex:cech-del}, directly implies \cref{corollary-b}.
\begin{corollary}
  \label{fullyreduced-supported-descending}
  Let $f$ be a generalized discrete Morse function and 
  $(\sigma,\tau)$ a non-zero persistence pair of the simplexwise filtration induced by the $f$-lexicographic order.
  Let $r=f(\sigma)$ and $s=f(\tau)$ be the function values of $\sigma$ and of $\tau$, respectively.
  Then the lexicographically minimal cycle of $[\partial\tau]$ in the open sublevel set $\sublevelopen{f}{s}$, given as the column $R_\tau$ of the totally reduced filtration boundary matrix, is supported on the descending complex $\descpxf{f}{r}$.
  \end{corollary}
\begin{proof}
  As $(\sigma,\tau)$ is a non-zero persistence pair, we know that $f(\sigma)<f(\tau)=s$ and that $\tau$ is a critical simplex. As $R$ is a reduction of $D$, this implies that $R_\tau$ and $\partial \tau$ are homologous cycles in $\sublevelopen{f}{s}$.
  Since $R$ is totally reduced, the cycle $R_\tau$ does not contain a (non-essential) birth simplex of the (smaller) simplexwise filtration of $\sublevelopen{f}{s}$ induced by the $f$-lexicographic order.
  Thus, \cref{lexmin-iff-nobirthspx} implies that $R_\tau$ is the lexicographically minimal cycle of $[\partial \tau]$ in~$\sublevelopen{f}{s}$.
  As $r=f(\sigma)$ and $R$ is a reduction of $D$, the cycle $R_\tau$ is supported on the subcomplex $\sublevel{f}{r}$ of $\sublevelopen{f}{s}$, implying that it is also a lexicographically minimal cycle in $\sublevel{f}{r}$. It now follows from \cref{inv-chain-supp-on-desc-cpx} that the cycle $R_\tau$ is supported on the descending complex $\descpxf{f}{r}$.
\end{proof}

\bibliography{literature.bib}

\newpage
\appendix

\section{Algebraic Morse theory and persistence}

\subsection{Gradient pairs from persistence pairs}
\label{appendix:gradientpairs-persistentpairs}

In this section, we discuss how the direct sum decomposition of filtered chain complexes yields a straightforward interpretation of persistence pairs as an algebraic gradient, which was mentioned in \cref{gradientpairs-persistentpairs}.

Consider the direct sum decomposition explained in \cref{prelim-pers-hom} and equip the chain complex~$C_*$ with the new ordered basis $E_* = \eta_1<\dots < \eta_l$ given by
\[
\eta_i=
\begin{dcases*}
R_j & if there exists an index persistence pair $(i,j)$,\\
S_i & if $i$ is a death or essential index.
\end{dcases*}
\]
We call $E_*$ the \emph{decomposition basis}, noting that it induces a decomposition of the filtered chain complex as described in \cref{prelim-pers-hom}.
Note that with respect to the original basis we have $\chainpivot{\chainbasis{*}}{\eta_i}=\sigma_i$ for all $i$.
By pairing the death columns $\VMatrix_j$ with their boundaries $R_j=D\cdot \VMatrix_j$, we obtain a set of disjoint pairs that we call the \emph{decomposition gradient} of $S$:
\[
\{(R_j,\VMatrix_j)\mid\ j\text{ is a death index}\}.
\]

\begin{proposition}
 The decomposition gradient is an algebraic gradient on the basis $E_*$.
\end{proposition}
\begin{proof}
Consider the function $f\colon E_* \to \mathbb{N}$ with values $f(R_j)=f(\VMatrix_j)=j$ for every death index $j$ and $f(\VMatrix_i)=i$ for every essential index $i$.
Note first that for any death index $j$, the basis element $R_j$ is the only facet of $\VMatrix_j$, and $R_j$ has no facets as it is a cycle.
Similarly, for any essential index $i$, the basis element $\VMatrix_i$ has no facets, as well.
Thus, it is immediate that~$f$ is a monotonic function whose algebraic gradient is the decomposition gradient.
\end{proof}

\subsection{The flow of an algebraic gradient}
\label{appendix-alg-flow}
In this section, we provide the relegated proofs for the statements in \cref{section-alg-flow}.

\begin{proof}[Proof of \cref{comp-inv-chains-ref}.]
   Let $c\in C_n$ be a $\Phi$-invariant chain.
    Recall from \cref{def-alg-flow} that $\Phi$ is given by
   $\Phi(c)=c+\partial \Flow(c)+\Flow(\partial c)$,
  where $\Flow\colon C_*\to C_{*+1}$ is the linear map with $\Flow(\sigma)=-\IP{\partial \tau}{\sigma}^{-1} \cdot \tau$ if $(\sigma,\tau)\in V$ and $0$ on all other basis elements of $\chainbasis{*}$.
  By \cref{characterization-inv-chains} and linearity, we can assume without loss of generality that $c$ is of the form $\Phi^\infty(\eta)$ for a $V$-critical basis element~$\eta\in \chainbasis{n}$.
   By definition, $\Phi^\infty=\Phi^{r+1}$ for a large enough $r\in \mathbb{N}$, and by \cref{stabilized-critical}, the chain $c=\Phi^{r+1}(\eta)$ satisfies $\Flow(c)=\Flow(\Phi^{r+1}(\eta))=0.$ Hence, we have
  \[c=\Phi(c)=c+\Flow(\partial c)+\partial \Flow(c)=c+\Flow(\partial c)\]
  from which we conclude that $\Flow(\partial c)=0$ holds.
  By construction of $\Flow$, this implies that $c$ and~$\partial c$ are the sums of critical elements and gradient cofacets of $V$.
  
  Similarly to before, $\Psi$ is given~by $\Psi(c)=c+\partial {E}(c)+ {E}(\partial c)$,
  where ${E}\colon C_*\to C_{*+1}$ is the linear map with $ {E}(\sigma)=-\IP{\partial \tau}{\sigma}^{-1} \cdot \tau$ if $(\sigma,\tau)\in W$ and $0$ on all other basis elements of~$\chainbasis{*}$.
  As $W\subseteq V$, it holds true that $c$ and $\partial c$ are the sums of critical elements and gradient cofacets of $W$, as well.
By construction of $E$, this implies $E(c)=E(\partial c)=0$, and hence
  \[\Psi(c)=c+\partial {E}(c)+ {E}(\partial c)=c+ 0+ 0=c,\]
  meaning that $c$ is $\Psi$-invariant, proving the claim.
  \end{proof}

  \goodbreak

  The following statement is a direct consequence of the definitions.
  \begin{lemma}
    \label{sorting-implies-nreduced}
    Let $(C_*,\chainbasis{*})$ be a based chain complex and let $\sigma_1<\dots< \sigma_l$ be any total order on the basis $\chainbasis{*}$ corresponding to an elementwise filtration that refines the sublevel set filtration induced by an algebraic Morse function $f$.
  Then, with respect to this total order, for any gradient pair $(a,b)\in V$ of $f$ the pivot of $\partial b$ is equal to $a$, $\chainpivot{\chainbasis{*}}{\partial b}=a$.
   \end{lemma}

   Let $(C_*,\chainbasis{*})$ be a based chain complex and $V$ an algebraic gradient on $\chainbasis{*}$.
Let $\spacegradientfacets{n}$ be the linear subspace of $C_n$ spanned by the gradient facets of $V$ and consider the canonical linear projection $\projfacets{n}\colon \spann{\gradbound{n}{V}}\to \spacegradientfacets{n}$ with respect to the basis $\chainbasis{n}$.
  \begin{lemma}
    \label{gradbound-independent}
    The subset $\gradbound{n}{V}=\{\partial b\mid \exists\ (a,b)\in V \text{ with }a\in \chainbasis{n}\}\subseteq C_n$ is linearly independent, and the canonical linear projection 
    $\projfacets{n}\colon \spann{\gradbound{n}{V}}\to \spacegradientfacets{n}$
    is an isomorphism.
  \end{lemma}
  \begin{proof}
  By \cref{sorting-implies-nreduced}, we can take a total order $\sigma_1<\dots<\sigma_l$ on the basis $\chainbasis{*}$ such that for any pair $(a,b)\in V$ the pivot of $\partial b$ is equal to $a$, $\chainpivot{\chainbasis{*}}{\partial b}=a$.
  As the pairs in $V$ are disjoint, we know that the pivots $\chainpivot{\chainbasis{*}}{\partial b}=a$ for $(a,b)\in V$ are distinct.
In particular, this implies that the subset $\gradbound{n}{V}$ is linearly independent.
  
  To prove the second claim, note that, similarly to before, the pivots $\chainpivot{\chainbasis{*}}{\projfacets{n}(\partial b)}=a$ for $(a,b)\in V$ with $a\in\chainbasis{n}$ are distinct.
This implies that the linearly independent vectors $\gradbound{n}{V}$ are sent to linearly independent vectors in $\spacegradientfacets{n}$.
As $\spann{\gradbound{n}{V}}$ and~$\spacegradientfacets{n}$ have the same dimension $\card{(\{(a,b)\in V \mid a\in \chainbasis{n}\})}$, this shows that $\projfacets{n}$ is an~isomorphism.
  \end{proof}

  \begin{proof}[Proof of \cref{phiinv-equ-nogradientfacets}.]
    We start by proving the first claim.
    On the one hand, we know from \cref{characterization-inv-chains} and \cref{stabilized-critical}, that a $\Phi$-invariant chain contains no gradient facets of $V$.
    On the other hand, if a cycle contains no gradient facets of $V$, then it is $\Phi^{(a,b)}$-invariant for every pair~$(a,b)\in V$, and it follows from \cref{inv-iff-vecflow-inv} that it is also $\Phi$-invariant.
    
    To prove the second claim, let $z'\in Z_n$ be any $\Phi$-invariant cycle that is $V$-homologous to~$z$.
    Then by definition, there exists a $\partial e\in \spann{\gradbound{n}{V}}$ with $z-z'=\partial e$.
    Recall from \cref{gradbound-independent}, that the projection $\projfacets{n}\colon \spann{\gradbound{n}{V}}\to \spacegradientfacets{n}$ 
    onto the subspace of $C_n$ spanned by the gradient facets of $V$ is an~isomorphism.
    By assumption, the difference $z-z'=\partial e$ is $\Phi$-invariant, and therefore it contains no gradient facets of $V$, by our first claim.
    Thus, we have $\projfacets{n}(\partial e)=0$, implying $z-z'=\partial e=0$, and proving the second claim $z=z'$.
  \end{proof}

  Let $(C_*,\chainbasis{*}=\sigma_1<\dots<\sigma_l)$ be a based chain complex with an elementwise filtration.
Motivated by \cref{sorting-implies-nreduced}, we make the following definition.

  \begin{definition}
   We call an algebraic gradient $V$ on $\chainbasis{*}$ \emph{reduced} if for all pairs $(a,b)\in V$ we have $\chainpivot{\chainbasis{*}}{\partial b}=a$, and we call it \emph{reduced in degree $n$}, or \emph{$n$-reduced}, if this holds for all pairs $(a,b)\in V$ with~$a\in \chainbasis{n}$.
  \end{definition} %

  Motivated by the notion of two cycles being $V$-homologous, and by \cref{phiinv-equ-nogradientfacets}, we make the following definition.
\begin{definition}
  \label{reductiongradient-complete}
  We say that an algebraic gradient $V$ on $\chainbasis{*}$
  \emph{generates the $n$-boundaries}~if the set
  of gradient cofacet boundaries 
  $\gradbound{n}{V}$ generates the entire subspace of boundaries $B_n\subseteq C_n$.
 \end{definition}

 \begin{remark}
   \label{reduction-reduced-complete}
 By construction, any reduction gradient is $n$-reduced and generates the $n$-boundaries for all $n$.
 \end{remark}

\subsection{Relating the algebraic flow and lexicographically minimal cycles}
\label{appendix-flow-equ-lex}
In this section, we prove \cref{invpivotflow-equ-lexminsimplicial,unique-minimizers} through a sequence of general statements.
As we deal with two different bases in \cref{invpivotflow-equ-lexminsimplicial}, we first introduce the \emph{$n$-reduction gradient}, which provides an intermediate between lexicographically minimal cycles with respect to $\chainbasis{*}$ and the reduction gradient defined on the reduction basis $\Omega_*$.

Let $(C_*,\chainbasis{*}=\sigma_1<\dots<\sigma_l)$ be a based chain complex with an elementwise filtration, and let $R=D\cdot \VMatrix$ be a reduction of the filtration boundary matrix. %
We equip the chain complex $C_*$ with the new ordered basis $\Pi_*=\tau_1<\cdots<\tau_l$ given by
\[
\tau_i=
\begin{dcases*}
S_i & if $i$ is a death index and $\sigma_i\in C_{n+1}$,\\
\sigma_i & otherwise.
\end{dcases*}
\]
We call $\Pi_*$ the \emph{$n$-reduction basis}.
Note that $\Sigma_*$ and $\Pi_*$ induce the same lexicographic preorder on $C_n$.
In particular, note that for every death index $j$ with $\sigma_j\in C_{n+1}$ we have $\chainpivot{\Pi_*}{R_j}=\chainpivot{\chainbasis{*}}{R_j}$, where $R_j=D\cdot \VMatrix_j$.
Similarly to \cref{pivot-gradient-is-gradient}, the collection of disjoint pairs
\[\{(\chainpivot{\Pi_*}{R_j},\VMatrix_j)\mid\ j\text{ is a death index and }\sigma_j\in C_{n+1}\}\]
is an algebraic gradient on $\Pi_*$, the \emph{$n$-reduction gradient} of $S$.
  By construction, the $n$-reduction gradient is $n$-reduced and generates the $n$-boundaries.

We now prove some general statements about algebraic gradients that are $n$-reduced and generate the $n$-boundaries.
Let $(C_*,\Pi_*=\tau_1<\dots<\tau_l)$ be any based chain complex with an elementwise filtration.
\begin{lemma}
\label{irred-nogradientfacets}
If $V$ is an $n$-reduced algebraic gradient on $\Pi_*$ that generates the $n$-boundaries, then a cycle $z\in Z_n$ is reducible if and only if its support in the basis $\Pi_*$ contains a gradient facet of $V$.
\end{lemma}
\begin{proof}
 Assume that the cycle $z$ is reducible.
By definition, there exists a cycle $z'\sqsubset z$ and a chain $e\in C_{n+1}$ such that $z-z'=\partial e$.
As $V$ generates the $n$-boundaries, the boundary $\partial e \in B_n=\spann{\gradbound{n}{V}}$ can be written as a linear combination $\partial e=\sum \lambda_i \partial b_i$ for some $(a_i,b_i)\in V$ with~$\lambda_i\neq 0$ and $a_i\in\Pi_{n}$.
Since the chains $\partial b_i$ have distinct pivots, as $V$ is $n$-reduced and the pairs in $V$ are disjoint, there exists a unique chain $\partial b_j$ with $\chainpivot{\Pi_*}{\partial e}=\chainpivot{\Pi_*}{\partial b_j}=a_j$.
As $\lambda_j\neq 0$, the gradient facet $a_j$ must be contained in $z$, as otherwise $z\sqsubseteq z'$, contradicting the assumption.

 For the converse, note that if the cycle $z$ contains some gradient facet $a$ with $(a,b)\in V$, then $z'=\Phi^{(a,b)}(z)=z+\IP{z}{a}\cdot \partial \Flow(a)$, where $\Flow(a)=-\IP{\partial b}{a}^{-1} \cdot b$, is a homologous cycle with~$z'\sqsubset z$, since $\chainpivot{\Pi_*}{\partial b}=a$ and $a$ is contained in $\supp{{\Pi_*}}{z}$ but not in~$\supp{\Pi_*}{z'}$.
\end{proof}

\begin{proposition}
\label{flowinv-equ-lexmin}%
If $V$ is an $n$-reduced algebraic gradient on a basis $\Pi_*$ that generates the $n$-boundaries, then for a cycle $z\in Z_n$ the following are equivalent:
\begin{itemize}
 \item $z$ is lexicographically minimal with respect to the basis $\Pi_*$;
 \item $z$ is invariant under the algebraic flow $\Phi$ determined by $V$, i.e., it satisfies $\Phi(z)=z$.
\end{itemize}
\end{proposition}
\begin{proof}
Let $z\in Z_n$ be a $\Phi$-invariant cycle.
By \cref{phiinv-equ-nogradientfacets}, this is equivalent to $z$ not containing a gradient facet, and that, in turn is equivalent by \cref{irred-nogradientfacets} to $z$ being irreducible, which is the same as being lexicographically minimal.
\end{proof}

In particular, note that \cref{flowinv-equ-lexmin} implies that every $n$-reduced algebraic gradient on $\Pi_{*}$ that generates the $n$-boundaries has the same set of invariant cycles.

\begin{proof}[{Proof of \cref{invpivotflow-equ-lexminsimplicial}}]
  By construction of the algebraic flow and under the assumption that $\VMatrix$ is death-compatible, the reduction flow agrees on cycles with the flow determined by the $n$-reduction gradient on the $n$-reduction basis $\Pi_*$.
  Thus, together with \cref{flowinv-equ-lexmin}, this implies that the cycle $z$ is invariant under the reduction flow if and only if $z$ is lexicographically minimal with respect to the ordered basis $\Pi_*$.
  As mentioned before, the ordered bases $\Pi_*$ and $\chainbasis{*}$ induce the same lexicographic preorder on $C_n$, and hence the claim follows.
\end{proof}

Next, we show the uniqueness of optimal homologous cycles:
\begin{proposition}
  \label{unique-minimizers}
   Let $(C_*,\chainbasis{*}=\sigma_1<\dots<\sigma_l)$ be a based chain complex with an elementwise filtration, and 
   let $z$ be lexicographically minimal cycle with respect to the basis $\chainbasis{*}$.
   If $c\in Z_n$ is homologous to $z$ and~$c\sqsubseteq z$ with respect to $\chainbasis{*}$, then we must have~$c=z$.
  \end{proposition}
  
\begin{proof}%
  Let $R=D\cdot \VMatrix$ be a reduction of the filtration boundary matrix, and let $V$ be the $n$-reduction gradient on the $n$-reduction basis $\Pi_*$.
  Recall that $V$ is  $n$-reduced and generates the $n$-boundaries by construction.
  As mentioned before, the ordered bases $\Pi_*$ and $\chainbasis{*}$ induce the same lexicographic preorder on $C_n$, and thus, both cycles $z$ and $c$ are also lexicographically minimal with respect to the basis $\Pi_*$.
  By \cref{flowinv-equ-lexmin}, this implies that both cycles are invariant under the flow determined by $V$.
  As $V$ generates the $n$-boundaries, for $z$ and $c$ being homologous is the same as being $V$-homologous, and hence \cref{phiinv-equ-nogradientfacets} implies $c=z$.
  \end{proof}

  Finally, we give a proof of \cref{lexmin-iff-nobirthspx} using these results.
  \begin{proof}[{Proof of \cref{lexmin-iff-nobirthspx}}]
    By \cref{reduction-deathcomp}, we know that there exists a death-compatible reduction matrix $\VMatrix$.
Thus, by \cref{invpivotflow-equ-lexminsimplicial}, a cycle $z$ is lexicographically minimal with respect to the elementwise filtration if and only if $z$ is invariant under the reduction flow.
As in the proof of \cref{invpivotflow-equ-lexminsimplicial}, this is equivalent to $z$ being invariant under the
    flow of the $n$-reduction gradient.
    This in turn is, by \cref{phiinv-equ-nogradientfacets}, equivalent to $z$ not containing any gradient facets of the $n$-reduction gradient.
Finally, by the construction of the $n$-reduction gradient, this is equivalent to $z$ not containing any (non-essential) birth elements, meaning that it only contains death elements and essential elements.
  \end{proof}

\subsection{Algebraic flow as matrix reduction}
\label{appendix-algflow-matrixred}

 We demonstrate how the algebraic flow on a cycle can be interpreted as a variant of Gaussian elimination, tying it closely to the exhaustive reduction introduced in~\cref{prelim-pers-hom}.
 
 Let $(C_*,\chainbasis{*}=\sigma_1<\dots<\sigma_l)$ be a based chain complex with an elementwise filtration, and let $V$ be an algebraic gradient on $\chainbasis{*}$.
 
 \begin{figure}[h!]
 \begin{algorithm}[H]
 \caption{Gradient flow reduction}\label{alg:gradient-flow}
 \KwIn{$D=\partial$ the $l\times l$ filtration boundary matrix, $c$ a cycle}
     \For{$i=1,\dots,l$}
     {
         \If{$c_i\neq 0$ and $(\sigma_i,\sigma_j)\in V$ a gradient pair}{
         $\mu=-c_i/D_{i,j}$\;
         $c=c + \mu\cdot D_j$\;
         }
     }
   \Return $c$
 \end{algorithm}
 \end{figure}
 
 \goodbreak 

 \begin{proposition}
 \label{reduction-computes-flow}
 If $V$ is $n$-reduced and $c\in Z_n$ a cycle, then \cref{alg:gradient-flow} computes $\Phi(c)$.
 \end{proposition}
 \begin{proof}
 By definition, for any cycle $c$ the algebraic flow is given by
 $\Phi(c)=c+\partial \Flow(c)$.
 We denote by $L$ the set of pairs $(i,j)$ with $c_i\neq 0$ and $(\sigma_i,\sigma_j)\in V$.
 Recall from \cref{def-alg-flow} that $\Flow\colon C_*\to C_{*+1}$ is the linear map with $\Flow(\sigma_i)=-\IP{\partial \sigma_j}{\sigma_i}^{-1} \cdot \sigma_j$ if $(\sigma_i,\sigma_j)\in V$ and $0$ on all other basis elements.
 Note that $\partial\sigma_j$ is represented by the column $D_j$ and $\IP{\partial \sigma_j}{\sigma_i}=D_{i,j}$.
Thus, we have
 \[\partial \Flow(c)=\sum\limits_{(i,j)\in L} c_i\cdot \partial\Flow(\sigma_i)=\sum\limits_{(i,j)\in L} (-c_i\IP{\partial \sigma_j}{\sigma_i}^{-1})\cdot \partial\sigma_j=\sum\limits_{(i,j)\in L} (-c_i/D_{i,j})\cdot D_j.\]
 By assumption of $V$ being $n$-reduced, these columns $D_j$ have distinct pivots and thus, traversing and updating the cycle $c$ from small to large index has the same effect as adding a summand from the sum above to the cycle $c$ one after the other.
 Hence, \cref{alg:gradient-flow} computes~$\Phi(c)$.
 \end{proof}
 
 To compute the image of a cycle under the stabilized flow, one can use \cref{alg:stabilized-flow}, which resembles the exhaustive reduction from \cref{prelim-pers-hom}, and that does not require any additional specific choices besides the algebraic gradient, in contrast to~\cref{alg:gradient-flow}.
 
 \begin{figure}[h!]
 \begin{algorithm}[H]
 \caption{Stabilized flow reduction}\label{alg:stabilized-flow}
 \KwIn{$D=\partial$ the filtration boundary matrix, $c$ a cycle}
     \While{ there exists $i$ with $c_i\neq 0$ and $(\sigma_i,\sigma_j)\in V$ a gradient pair}
     {
         $\mu=-c_i/D_{i,j}$\;
         $c=c + \mu\cdot D_j$\;
     }
   \Return $c$
 \end{algorithm}
 \end{figure}

 \begin{proposition}
 For any cycle $c\in Z_n$, \cref{alg:stabilized-flow} computes~$\Phi^\infty(c)$.
 \end{proposition}
 \begin{proof}
 Note that we can choose a different elementwise filtration by the $\sigma_i$ without crucially affecting \cref{alg:stabilized-flow}.
Moreover, by \cref{sorting-implies-nreduced}, we can assume that $V$ is $n$-reduced.
 Thus, the algorithm is essentially the same as \cref{alg:exhaustive}, implying that it terminates.
 
Now note that \cref{alg:stabilized-flow} computes a $V$-homologous cycle $c'$ of $c$ that contains no gradient facets, and thus it is $\Phi$-invariant by \cref{phiinv-equ-nogradientfacets}.
 By definition, the stabilized cycle $\Phi^\infty(c)$ is also a $\Phi$-invariant cycle and $V$-homologous to $c$.
 Therefore, both $c'$ and $\Phi^\infty(c)$ are $\Phi$-invariant cycles that are $V$-homologous.
 \Cref{phiinv-equ-nogradientfacets} now implies $c'=\Phi^\infty(c)$.
 \end{proof}

  \section{Reduction chains and descending complexes}
  \label{appendix-reduction-chains}

  For a generalized discrete Morse function $f\colon K\to \mathbb{R}$, we know from \cref{fullyreduced-supported-descending} that the columns of the totally reduced filtration boundary matrix corresponding to non-zero persistence pairs are supported on the descending complex. 
  In this section, we show that the reduction chains, i.e., the chains given as the columns of the reduction matrix, that correspond to essential and non-zero persistence pairs are also supported on the descending~complex.

  Let $D$ be the filtration boundary matrix of the simplexwise filtration of $K$ induced by the $f$-lexicographic order, and let $\VMatrix$ be a death-compatible reduction matrix with $R=D\cdot \VMatrix$ totally reduced.
  
  \goodbreak 
\begin{proposition}
  \label{reductioncolumn-ondescending}
  Let $f\colon K\to \mathbb{R}$ be a generalized discrete Morse function and let $\tau\in K$ be a 
  critical simplex, i.e., a simplex that is either essential or contained in a non-zero persistence pair of the simplexwise filtration induced by the $f$-lexicographic order. Let $r=f(\tau)$ be the function value of $\tau$.
  Then the chain given as the column $\VMatrix_\tau$ of the reduction matrix is supported on the descending complex $\descpxf{f}{r}$.
\end{proposition}
\begin{proof}
 Note that if $\tau$ is an essential simplex or a birth simplex contained in a non-zero persistence pair, then it is an essential simplex of the (smaller) simplexwise filtration of $\sublevel{f}{r}$ induced by the $f$-lexicographic order.
Let $W_r$ be the zero persistence apparent pairs gradient of this simplexwise filtration and denote by $\Phi\colon C_*(\sublevel{f}{r})\to C_*(\sublevel{f}{r})$, with $\Phi(c)=c+\Flow(\partial c)+\partial \Flow(c)$, the associated flow.

Assume first that $\tau$ is an essential simplex of the simplexwise filtration of $\sublevel{f}{r}$.
Then, the chain $c$ given as the column $\VMatrix_\tau$ of the reduction matrix is a cycle, i.e., $\partial c=0$.
Moreover, since $\VMatrix$ is death-compatible, we also have $\Flow(c)=0$.
Hence, we get 
\[\Phi(c)=c+\Flow(\partial c)+\partial \Flow(c)=c+0+0=c\]
and the chain~$c$ is $\Phi$-invariant.
Analogously as in the proof of \cref{inv-chain-supp-on-desc-cpx}, we conclude that $c$ is supported on the descending complex $\descpxf{f}{r}$.

Now assume that $\tau$ is a death simplex of the simplexwise filtration of $\sublevel{f}{r}$ that is contained in a non-zero persistence pair $(\sigma,\tau)$, implying that $\tau$ is not contained in a zero-persistence apparent pair.
By \cref{def:app-pairs}, and as the matrix $R$ is totally reduced, the column $R_\tau=D\cdot \VMatrix_\tau$ does not contain any apparent facet of the simplexwise filtration. %
Therefore, the chain $c$ given as the column $\VMatrix_\tau$ of the reduction matrix satisfies $\Flow(\partial c)=0$.
As before, since $\VMatrix$ is death-compatible, we also have $\Flow(c)=0$.
Hence, we get 
\[\Phi(c)=c+\Flow(\partial c)+\partial \Flow(c)=c+0+0=c\] 
and the chain~$c$ is $\Phi$-invariant.
Analogously to before, we conclude that $c$ is supported on the descending complex $\descpxf{f}{r}$, proving the claim.
\end{proof}

\begin{figure}[t]
  \centering
  \includegraphics[width=.2\linewidth]{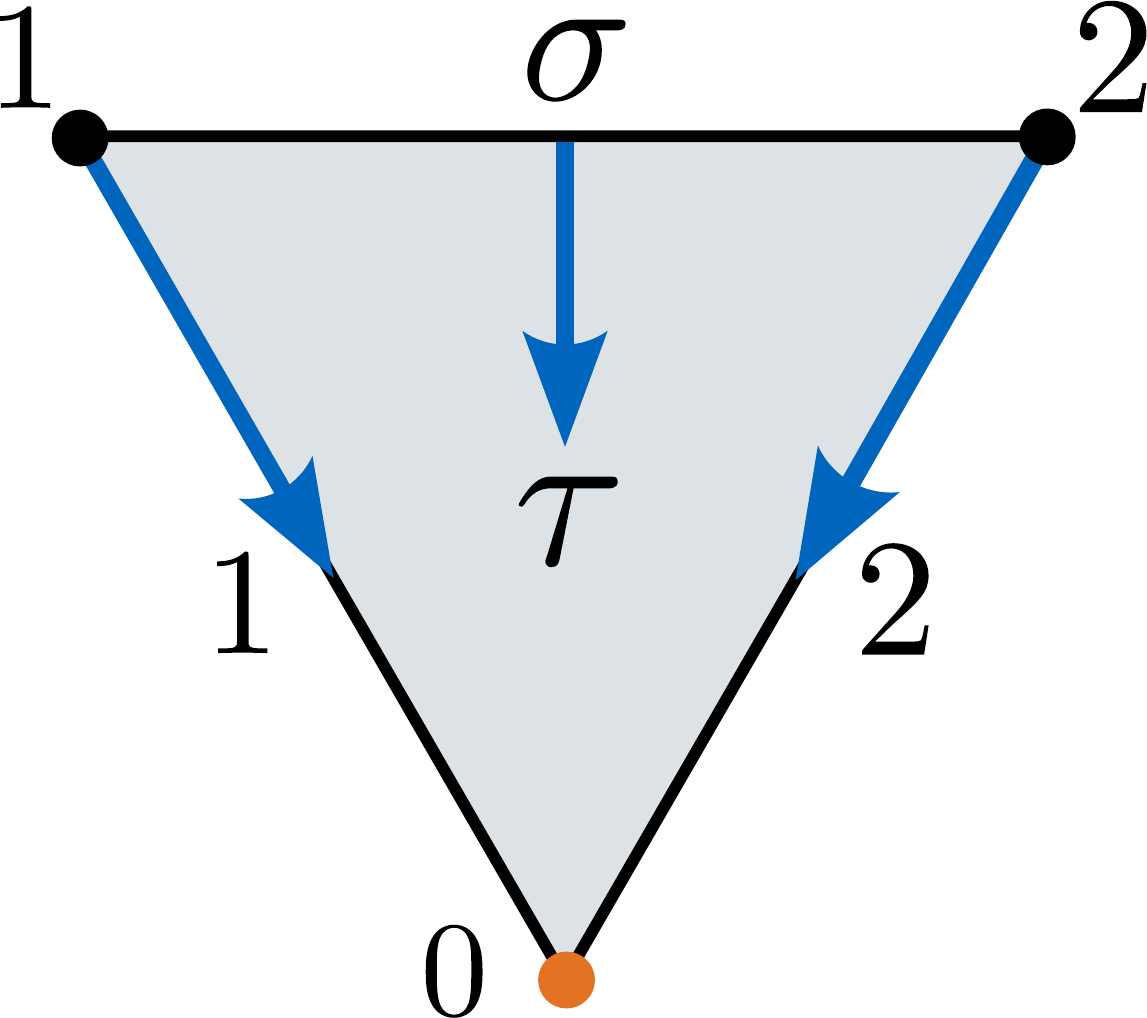}
  \caption{Discrete gradient (blue) together with the unique critical simplex (orange).}
  \label{fig:discrete-gradient}
\end{figure}
\begin{remark}
  Note that the assumption in \cref{fullyreduced-supported-descending,reductioncolumn-ondescending} on the pair $(\sigma,\tau)$ to be a non-zero persistence pair can not be dropped.
Consider, for example, the simplicial complex in \cref{fig:discrete-gradient} and the discrete Morse function $f\colon K\to\mathbb{R}$ that assigns to the simplices $\sigma$ and $\tau$ the value $3$ and to all other simplices the value as indicated.
  Then the pair $(\sigma,\tau)$ is a zero persistence apparent pair, but neither $S_\tau=\tau$ nor $S_\sigma=R_\tau=\partial\tau$ is supported on the descending complex $\descpxf{f}{3}$, which only consists of the orange vertex with function value~$0$.
\end{remark}

\end{document}